\documentclass[12pt]{article}
\usepackage{amsmath}
\usepackage{amsthm}
\usepackage{graphicx}
\usepackage{amssymb}
\usepackage{amscd}
\usepackage[all]{xy}
\theoremstyle{plain}  
  \newtheorem{thm}{Theorem}[section]
  \newtheorem{cor}[thm]{Corollary}   
  \newtheorem{lem}[thm]{Lemma}    
  \newtheorem{prop}[thm]{Proposition}

\theoremstyle{definition}  
 \newtheorem{defn}[thm]{Definition}
\theoremstyle{remark}
  \newtheorem*{rem}{Remark}

  \newtheorem*{ack}{Acknowledgements}
\makeatletter
    
    \@addtoreset{equation}{section}
  \makeatother
\newcommand{\C}{\mathbb C} 
 
\newcommand{\Z}{\mathbb Z} 
\newcommand{\N}{\mathbb N}  
\newcommand{\R}{\mathbb R}

\newcommand{\HH}{{\bf H}}

\newcommand{\hC}{\hat \C}

\newcommand{\D}{\mathcal D}

\newcommand{\M}{\mathcal M} 
\newcommand{\B}{\mathcal B}  
  
\newcommand{\LL}{\mathcal L}

\newcommand{\A}{\mathcal A} 
 
\newcommand{\hN}{{\hat N}}
\newcommand{\hP}{{\hat P}}

\newcommand{\lam}{\lambda} 
\newcommand{\alp}{\alpha} 
\newcommand{\Lam}{\Lambda} 
 
\newcommand{\fty}{\infty} 
\newcommand{\ep}{\epsilon}

\newcommand{\bd}{\partial} 
\newcommand{\wh}{\widehat}
\newcommand{\wt}{\widetilde}

\newcommand{\sm}{\setminus}
\newcommand{\la}{\langle}
\newcommand{\ra}{\rangle}
\newcommand{\ze}{\zeta}

\newcommand{\vphi}{\varphi}

\newcommand{\tr}{{\mathrm{tr}}}
\newcommand{\Tr}{{\mathrm{Tr}}}
\newcommand{\Hom}{{\mathrm{Hom}}}

\newcommand{\inte}{\mathrm{int}}
\newcommand{\id}{\mathit{id}}
  
\newcommand{\im}{\mathrm{Im}} 
\newcommand{\re}{\mathrm{Re}} 
\newcommand{\psl}{{\mathrm{PSL}}(2,\mathbb C)}
\newcommand{\ssl}{{\mathrm{SL}}(2,\mathbb C)}

\newcommand{\mz}{{\mu,\ze}}

\title{Linear slices close to a Maskit slice} 
\author{Kentaro Ito}
\begin{document}

\maketitle

\begin{abstract}
We consider linear slices of the space of Kleinian once-punctured torus groups; 
a linear slice is obtained by fixing the value of the trace of one of the generators.  
The linear slice for trace $2$ is called the Maskit slice. 
We will show that if traces converge `horocyclically' to $2$ 
then associated linear slices 
converge to the Maskit slice, whereas if the traces converge 
`tangentially' to $2$ the linear slices converge to a proper subset of the Maskit slice. 
This result will  be also rephrased in terms of complex Fenchel-Nielsen coordinates. 
In addition, we will show that there is a linear slice which is not  locally connected. 
\end{abstract}

\section{Introduction}

One of the central issues in the theory of Kleinian groups is to understand the structures 
of deformation spaces of Kleinian groups. 
In this paper we consider Kleinian punctured torus groups, 
one of the simplest classes of Kleinian groups with a non-trivial deformation theory. 

Let $S$ be a once-punctured torus and let $R(S)$ be the space of conjugacy classes of representations 
$\rho:\pi_1(S) \to \psl$ which takes a loop surrounding the cusp to a parabolic element. 
The space $AH(S)$ of Kleinian punctured torus groups is the subset of $R(S)$  of 
faithful representations with discrete images. 
Although the interior of $AH(S)$  
is parameterized by a product of Teichm\"{u}ller spaces of $S$, 
its boundary is quite complicated. 
For example, McMullen \cite{Mc2} showed that $AH(S)$ self-bumps,  
and Bromberg \cite{Br} showed 
that $AH(S)$ is not even locally connected. 
We refer the reader to \cite{Ca}  
for more information on the topology of 
deformation spaces of general Kleinian groups. 

In this paper we investigate the shape of $AH(S)$ form the point of view of the trace coordinates. 
Let us fix a pair $a,b$ of generators of $\pi_1(S)$. 
Then every representation $\rho$ in $R(S)$ is essentially 
determined by the data $(\tr \,\rho(a),\tr\,\rho(b))=(\alp,\beta)  \in \C^2$. 
Thus we identify $R(S)$ with $\C^2$ in this introduction (see Section 2 for more accurate treatment). 
We want to understand when $(\alp,\beta) \in \C^2$ corresponds to a point of  $AH(S)$. 
More precisely, we consider in this paper the shape the {\it linear slice}  
$$
\LL(\beta):=\{\alp \in \C : (\alp,\beta) \in AH(S)\}
$$
of $AH(S)$ 
when $\beta$ close to $2$. 
Note that $\LL(2)$ is known as the Maskit slice, corresponding to the set of representations 
$\rho \in AH(S)$ such that $\rho(b)$ is parabolic. 
It is natural to ask the following question:  
``When $\beta$ tends to $2$,  does $\LL(\beta)$ converge to $\LL(2)$?" 
Parker and Parkkonen \cite{PP}  studied this question 
in the case that  a real number $\beta>2$ tends to $2$, 
and obtained an affirmative answer for this case.  
In this paper, we consider the question above in the general case 
that  a complex number $\beta  \in \C \sm [-2,2]$ tends to $2$, 
and obtain the complete answer  to this question. 
In fact, the answer depends on the manner how $\beta$ tends to $2$. 

To describe our results, we need to introduce the notion of complex length.  
Let  $\rho \in R(S)$ and assume that 
$\beta=\tr \rho(b)$ is close to $2$. 
Then the complex length $\lam$ of $\rho(b)$ is determined by the relation 
$
\beta=2\cosh (\lam/2)
$
 and the normalization $\re\,\lam>0,\,\im\,\lam \in (-\pi,\pi]$.
  We denote this $\lam$ by $\lam(\beta)$. 
  Note that $\beta \to 2$ if and only if $\lam(\beta) \to 0$. 
We say that a sequence $\beta_n \in \C \sm[-2,2]$ 
converges {\it horocyclically}  to $2$ if for any disk 
in the right-half plane  $\C_+$ touching at zero,  
$\lam(\beta_n)$ are eventually contained in this disk. 
On the other hand, 
we say that  the sequence $\beta_n$ converges {\it tangentially} to $2$   
if there is a disk in $\C_+$ touching at zero 
which does not contain any $\lam(\beta_n)$.  
Now we can state our main result.  (See Theorems \ref{tan} and \ref{horo} for more precise statements. 
See also Figure \ref{tr_conv}.) 

\begin{thm}
Suppose that a sequence $\beta_n \in \C \sm [-2,2]$ converges to $2$. 
If  $\beta_n \to 2$ horocyclically, then $\LL(\beta_n)$ converge to $\LL(2)$ in the sense of 
Hausdorff. 
On the other hand, if $\beta_n \to 2$ tangentially, 
then $\LL(\beta_n)$ converge (up to subsequence) to a proper subset of $\LL(2)$ in the sense of Hausdorff. 
\end{thm}

 We now  sketch the essential idea which is underlying this phenomenon.  
 Especially we explain the reason why  
  the limit of linear slices is a proper subset of $\LL(2)$ 
  in the case where $\beta_n \to 2$ tangentially. 
  
  Now suppose that $\beta_n \to 2$ tangentially,  
  and that a sequence $\alp_n \in \LL(\beta_n)$ converges to $\alp \in \C$. 
  We will explain that $\alp$ should lie in a proper subset of $\LL(2)$. 
Let us take a sequence $\rho_n \in AH(S)$ such that 
$(\tr \,\rho_n(a),\tr\,\rho_n(b))=(\alp_n,\beta_n)$. 
 Since $(\alp_n, \beta_n) \to (\alp,2)$ as $n \to \fty$,  
 and since $AH(S)$ is closed, we have $(\alp,2) \in AH(S)$,  and hence $\alp \in \LL(2)$.  
By taking conjugations, we may assume that $\rho_n(a) \to A_\alp$ and 
$\rho_n(b) \to B$ in $\psl$, where 
$$
A_\alp=\left(\begin{array}{cc}
 \alp & -i \\ 
-i & 
0 \end{array}\right) 
\quad \text{and} \quad 
B=\left(\begin{array}{cc}
1  & 2 \\ 
0 & 1 
\end{array}\right). 
$$
In addition, by pass to a subsequence if necessary, we may also assume that the sequence 
$\rho_n(\pi_1(S))$ converges geometrically to 
a Kleinian group $\Gamma$, which contains the algebraic limit $\la A_\alp, B\ra$. 
From the assumption that  $\beta_n \to 2$ tangentially,  one can see that 
the cyclic groups $\la\rho_n(b) \ra$ converge geometrically to rank-$2$ abelian group 
 $\la B, C \ra$,  where $C$ is of the form 
$$
C=\left(\begin{array}{cc}
 1 &  \ze\\ 
0 & 1 
\end{array}\right)
$$
for some $\ze \in \C \sm \R$, see Theorem \ref{cyclic}. 
Therefore the geometric limit $\Gamma$ contains the group $\la A_\alp,B, C\ra$. 
For any given integer $k$, one see from $C^{k} A_\alp=A_{\alp-k i\ze}$ that 
the group $\la A_{\alp-ki\ze}, B\ra$ is a subgroup of the  Kleinian group $\Gamma$. 
Hence the group $\la A_{\alp-ki\ze}, B\ra$ is discrete and thus $\alp -ki\ze \in \LL(2)$. 
Therefore $\alp$ should be contained in the intersection 
$$
\bigcap_{k\in \Z} (k i\ze+\LL(2)),
$$ 
which is a proper subset of $\LL(2)$. 

In the proof of Theorem 1.1, 
we will make an essential use of Bromberg's theory in \cite{Br}. 
In fact, Bromberg obtained in \cite{Br} a coordinate system for representations in $AH(S)$ close to the 
Maskit slice. 
The poof of Theorem 1.1 is then obtained by comparing Bromberg's coordinates and the trace coordinates. 

Some other topics and computer graphics of linear slices  can be fond in 
\cite{Mc2}, \cite{MSW} and \cite{KY}, as well as \cite{PP}.  

This paper is organized as follows;
In section 2, we recall some basic fact about spaces of representations and their subspaces. 
In section 3, we introduce the trace coordinates for  the space $R(S)$ 
of representations  of the once-punctured torus group.  
In section 4, we recall Bromberg's theory in \cite{Br} which gives us a local model of the space $AH(S)$ 
of Kleinian once-punctured torus groups near the Maskit slice. 
In section 5, we consider relation between Bromberg's coordinates and the trace coordinates, 
and obtain an estimate which will be used in the proofs of the main results. 
We will show our main results, Theorems \ref{tan} and \ref{horo},  in section 6. 
We also show that there is a  linear slice 
which is not locally connected. 
In section 7, we translate our main results in terms of the complex Fenchel-Nielsen coordinates. 

The following is the mainstream of this paper,  where the top (resp. bottom) line is corresponding to the 
tangential (resp. horocyclic) convergence:  
\begin{eqnarray*}
\xymatrix{ 
\fbox{\text{Theorem \ref{estimate}}} \ar@{=>} [r] \ar@{=>}[dr]
& \fbox{\text{Proposition \ref{ball1}}} \ar@{=>}[r] 
& \fbox{\text{Lemma \ref{tanlem}}} \ar@{=>}[r] 
& \fbox{\text{Theorem \ref{tan}}} \\ 
& \fbox{\text{Proposition \ref{ball2}}} \ar@{=>}[r] 
& \fbox{\text{Lemma \ref{horolem}}} \ar@{=>}[r] 
& \fbox{\text{Theorem \ref{horo}}}
}
\end{eqnarray*}

\begin{ack}
The author would like to thank Hideki Miyachi for his many helpful discussions. 
He is also grateful to Keita Sakugawa for developing a computer program drawing linear slices, 
which was very helpful to proceed this research.    All computer-generated figures of linear slices 
of this paper are made by this program. 
\end{ack}

\section{Spaces of representations}

In this section, we recall the definitions of spaces we will work with.

Let $(M, P)$ be a paired manifold;  that is,  $M$ is a compact, hyperbolizable  $3$-manifold with boundary 
and $P$ is a disjoint union of tori and annuli in $\bd M$.  Especially, every torus component 
of $\bd M$ is contained in $P$. 
Let 
$$
{\cal R}(M,P):=\Hom_P^\mathrm{irr} (\pi_1(M),\psl)
$$
denote the set of all type-preserving, irreducible representations of $\pi_1(M)$ into $\psl$. 
Here a representation $\rho: \pi_1(M) \to \psl$ is said to be {\it type-preserving} 
if $\rho(\gamma)$ is parabolic or identity for every $\gamma \in \pi_1(P)$. 
The space of representations 
$$
R(M,P):={\cal R}(M,P)/\psl
$$
is the set of all $\psl$-conjugacy classes $[\rho]$ of representations $\rho$ in ${\cal R}(M,P)$. 
We endow this space $R(M,P)$ with the algebraic topology; that is, 
a sequence $[\rho_n]$ converges to $[\rho]$ if there are  
representatives $\rho_n$ in $[\rho_n]$ and $\rho$ in $[\rho]$ such that for every $g \in \pi_1(M)$ 
the sequence $\rho_n(g)$ converges to $\rho(g)$ in $\psl$. 
The conjugacy class $[\rho]$ of a representation $\rho$ is also denoted by $\rho$ 
if there is no confusion. 
We are interested in the topological nature of the space 
$$
AH(M,P):=\{\rho \in R(M,P) : \rho \text{ is faithful, discrete}\}. 
$$
It is known by J{\o}rgensen \cite{Jo} that $AH(M,P)$ is closed in $R(M,P)$. 
Let $MP(M,P)$ denote the subset of $AH(N,P)$ 
consists of representations $\rho$ which are 
minimally parabolic (i.e., $\rho(g)$ is parabolic if and only if $g \in \pi_1(P)$) and geometrically finite. 
It is known by Marden \cite{Mar} and Sullivan \cite{Su} 
 that $MP(N,P)$ is equal to the interior of $AH(N,P)$ as a subset of $R(N,P)$. 
Recently, it was shown by Brock, Canary and Minsky \cite{BCM} that 
the closure of $MP(M,P)$ is equal to $AH(M,P)$. 

In this paper, we only consider the following three paired manifolds 
$$(N,P), \quad (N,P'),  \quad (\hN,\hP)$$ which are constructed as follows (see Figure \ref{paired}): 
\begin{figure}
\begin{center}
\includegraphics[height=4cm, bb=0 0 1319 486]{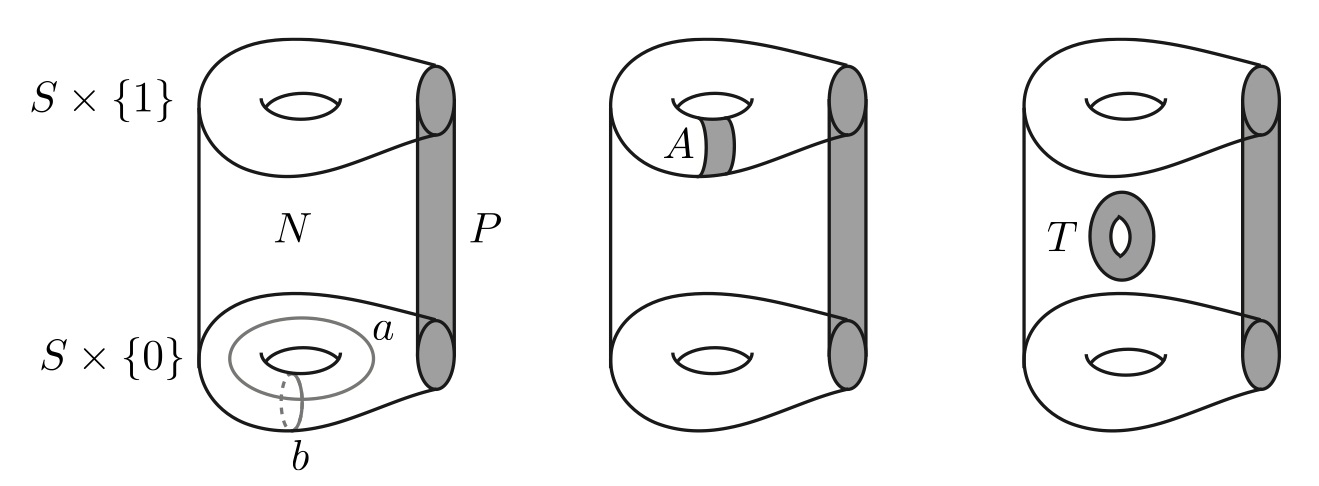}
\caption{ Paired manifolds $(N,P), (N,P')$ and $(\hN,\hP)$ (from left to right). }
\label{paired}
\end{center}
\end{figure}
Let $S$ be a torus with one open disk removed. 
Throughout of this paper, we fix a pair $a,b$ of generators of $\pi_1(S)$ 
such that the geometric intersection number equals one. 
Then the commutator $[a,b]=aba^{-1}b^{-1}$ is homotopic to $\bd S$. 
Now we set 
$$
N:=S \times [0,1]
$$
and 
$$
P:=\bd S \times [0,1].
$$
We next set $P':=P \cup A$, where $A \subset S \times \{1\}$  
is an annulus whose core curve 
is freely homotopic to $b \in \pi_1(S)$. 
Finally, we let 
\begin{eqnarray*}
(\hN,\hP):=(N \sm W, P \cup T), 
\end{eqnarray*}
where $W$ is a regular tubular neighborhood of $b \times \{1/2\}$ in $N=S \times [0,1]$ 
and $T:=\bd W$.

Note that $AH(N,P')$ lies in the boundary of $AH(N,P)$; in fact  
$\rho \in AH(N,P)$ lies in $AH(N,P')$ 
if and only if $\rho(b)$ is parabolic.  
This space $AH(N,P')$ is called the {\it Maskit slice} of $AH(N,P)$.  
It is known by Minsky \cite{Mi} that $AH(N,P')$ has exactly two connected components. 
Bromberg's theory in \cite{Br} gives us an information about the topology of $AH(N,P)$ near $AH(N,P')$. 
The aim of this paper is to understand the topology of $AH(N,P)$ near $AH(N,P')$ 
from the view point of the trace coordinates, which is explained in the next section.

\section{Trace coordinates for $AH(N,P)$}

In this section, we introduce a trace coordinate system 
on a subset of $R(N,P)$ containing $R(N,P')$.   

Recall that $(N,P)=(S \times [0,1],\bd S \times [0,1])$, 
where $S$ is a torus with one open disk removed.   
In this case,  the space $R(N,P)$ 
consists of all $\psl$-conjugacy classes of representations 
$$
\rho: \pi_1(S)=\la a,b \ra \to \psl
$$
which satisfy the condition $\tr (\rho([a,b]))=-2$. 
Note that the trace of the commutator $[a,b]$ is well defined, 
although the traces of $\rho(a)$ and $\rho(b)$ are determined up to sign. 

As we will see below, for any given $(\alp,\beta) \in \C^2$,  
there is  a representation $\rho \in R(N,P)$ which satisfies 
 $\tr^2\rho(a)=\alp^2, \, \tr^2\rho(b)=\beta^2$,  and this $\rho$ is  determined 
uniquely up to pre-composition of automorphism $(a,b) \mapsto (a,b^{-1})$ of $\pi_1(N)$. 
Therefore the subset  
$$
\D_{\tr}:=\{(\alp,\beta) \in \C^2 : \exists \rho  \in AH(N,P) 
\text{ s.t. } \tr^2\rho(a)=\alp^2, \, \tr^2\rho(b)=\beta^2\}
$$ 
of $\C^2$ is well-defined. 
Note that the set $\D_\tr$ is symmetric under the action $(\alp,\beta) \mapsto (\beta,\alp)$. 
For  a given $\beta \in \C$, the slice 
$$
\LL(\beta):=\{\alp \in \C :  (\alp,\beta) \in \D_{\tr}\}.
$$
of $\D_\tr$ is called the {\it linear slice} for $\beta$. 
Note that $\LL(\beta)$ is symmetric under the action of $z \mapsto -z$. 
The aim of this paper is  to understand the shape of $\LL(\beta)$ when $\beta$ is close to $2$. 

To study the shape of linear slices,  
it would be convenient  if we could identify $R(N,P)$ with $\C^2$ simply 
 by $\rho \mapsto (\tr\,\rho(a), \tr\,\rho(b))$. 
But the thing is not so simple. 
One reason is  that traces of $\rho(a), \rho(b)$ are determined up to sign, 
and the other reason is that, for a given $(\alp,\beta) \in \C^2$, 
there exist two candidate of representations $\rho$ which satisfy
$(\tr^2 \rho(a), \tr^2 \rho(b))=(\alp^2,\beta^2)$. 
Therefore, in this section, we will choose  an appropriate open domain $\Omega \subset R(N,P)$ 
so that there exists an embedding 
$\Tr:\Omega \to \C^2$  such that $\Tr(\rho)=(\alp,\beta)$ satisfies 
$(\tr^2 \rho(a), \tr^2 \rho(b))=(\alp^2,\beta^2)$ for every $\rho \in \Omega$.  

We begin by identifying $R(N,P')$ with $\C$. 
For a given $\alp \in \C$, let $\rho_\alp$ be the representation in $R(N,P')$ defined by 
$$
\rho_\alp(a):=
\left(\begin{array}{cc}
\alp  & -i \\ 
-i & 0
 \end{array}\right), \quad 
 \rho_\alp(b):=
 \left(\begin{array}{cc}
 1 & 2 \\ 
 0& 1
  \end{array}\right). 
$$
Then we have the following lemma.  (See Lemma 4.3 in \cite{Br}. 
Note that we are assuming that every element of $R(N,P')$ is irreducible.)

\begin{lem}\label{trace2}
The map $\psi:\C \to R(N,P')$ defined by $\alp \mapsto \rho_\alp$ is a homeomorphism. 
\end{lem}
Note that the map $\psi$ in Lemma \ref{trace2} induces a homeomorphism from $\LL(2)$ onto 
the Maskit slice $AH(N,P')$. 

In the next lemma, we will show that  the homeomorphism $\psi^{-1}:R(N,P') \to \C$ naturally extends to
 an embedding  from an open domain $\Omega \subset R(N,P)$ containing $R(N,P')$ 
into $\C^2$. 

\begin{lem}\label{trace} 
There exist an open,  connected, simply connected domain $\Omega \subset R(N,P)$ 
and a homeomorphism 
$$
{\mathrm{Tr}}:\Omega \to \C^2
$$ 
which satisfy the following: 
\begin{enumerate}
\item 
$\Omega$ contains $R(N,P')$,  
and $\Tr$ takes $R(N,P')$ onto $\C \times \{2\}$. 
In addition, we have  $\Tr(\rho_\alp)=(\alp,2)$ for every $\alp \in \C$. 
\item 
 For every $\rho \in \Omega$, $\Tr(\rho)=(\alp,\beta)$ satisfies  
$\tr^2 \rho(a)=\alp^2$ and $\tr^2 \rho(b)=\beta^2$. 
\end{enumerate}
\end{lem}

Throughout of this paper, we fix such a domain $\Omega$. 
We call $\Tr$ the {\it trace coordinate map} and 
$(\alp,\beta)=\Tr(\rho)$ the {\it trace coordinates} of $\rho \in \Omega$. 
The rest of this section is devoted to the proof of this lemma. 
The commutative diagram (3.2) should be helpful for understanding the arguments. 
The reader may skip this proof by admitting Lemma \ref{trace}. 

To show Lemma \ref{trace}, 
it is convenient to consider the space $\wt R(N,P)$ of 
representations of $\pi_1(N)$ into $\ssl$,  instead of $\psl$.  
More precisely, the set $\wt R(N,P)$ consists of 
$\ssl$-conjugacy classes of representations $\tilde \rho$ of 
$\pi_1(S)$ into $\ssl$ which satisfy the condition $\tr (\tilde\rho([a,b]))=-2$. 
The $\ssl$-conjugacy class of $\tilde \rho$ is also denoted by $\tilde \rho$ 
if there is no confusion. 
It is well known that an element $\tilde \rho$ 
of $\wt R(N,P)$ is uniquely determined by the triple 
$(\tr\tilde \rho(a),\tr\tilde \rho(b),\tr\tilde \rho(ab))$ of complex number 
(see for example \cite{Bo} or \cite{Go}): 

\begin{lem}\label{Bow}
The map 
$$
\wt \Tr: \wt R(N,P) \to \Xi:=\{(\alpha,\beta,\gamma)  \in \C^3 
: \alpha^2+\beta^2+\gamma^2=\alpha\beta\gamma\} \sm \{(0,0,0)\}
$$
defined by  $\tilde\rho \mapsto (\tr\tilde \rho(a),\tr\tilde \rho(b),\tr\tilde \rho(ab))$ 
is a homeomorphism. 
\end{lem}

By using this lemma, we often identify  $\wt R(N,P)$ with the subset $\Xi$ of $\C^3$. 
For $(\alpha,\beta) \in \C^2$, the numbers $\gamma$ satisfying 
$\alpha^2+\beta^2+\gamma^2=\alpha\beta\gamma$
are given by
\begin{equation*}
\gamma=\frac{1}{2}\left( \alpha \beta \pm \sqrt{\alpha^2\beta^2-4(\alpha^2+\beta^2)} \right). 
\end{equation*}
Therefore the projection 
$$
\Pi:\Xi \to \C^2 \sm\{(0,0)\}
$$
defined by $(\alp,\beta,\gamma) \mapsto (\alp,\beta)$ 
is a  two-to-one branched covering map. 
If we denote by $\gamma_1,\gamma_2$ 
the solutions of the equation $\alpha^2+\beta^2+\gamma^2=\alpha\beta\gamma$  
on $\gamma$, we have  $\gamma_1+\gamma_2=\alp\beta$. 
On the other hand, we have 
$$
\tr (AB)+\tr(AB^{-1})=\tr A \, \tr B
$$
for every $A,B \in \ssl$.  
Therefore one see that if two representations $\tilde \rho_1, \tilde \rho_2$ in $\wt R(N,P)$ 
have the same image under the map 
$\Pi \circ \wt \Tr$, 
they are only differing by pre-composition of the automorphism $(a,b) \mapsto (a,b^{-1})$ of $\pi_1(S)$. 

Now let 
$$
\pi:\wt R(N,P) \to R(N,P)
$$
be the natural projection, which is  a four-to-one  covering map. 
The group of covering transformation for $\pi$ is isomorphic to 
$\Z_2 \times \Z_2$ which is generated by $(\alp,\beta,\gamma) \mapsto (-\alp,\beta, -\gamma)$ 
and $(\alp,\beta,\gamma) \mapsto  (\alp,-\beta,-\gamma)$, 
where $\wt R(N,P)$ is identified with $\Xi \subset \C^3$ as in Lemma \ref{Bow}. 
 
Now let us take an open, connected and simply connected domain 
$\Delta \subset \C^2 \sm\{(0,0)\}$ 
which satisfy the following: 
\begin{enumerate}
\item $\Delta$ contains the set $\C \times \{2\}$, and 
\item $\Delta$ lies  in the set $\{(\alpha,\beta) \in \C^2 : \re\, \beta>0,\, 
\alpha^2\beta^2\ne 4(\alpha^2+\beta^2)\}$. 
\end{enumerate}
Here, the condition $\alpha^2\beta^2\ne 4(\alpha^2+\beta^2)$ is equivalent to the condition that 
the pair $(\alp,\beta)$ is not a critical value of the projection $\Pi:\Xi \to \C^2 \sm \{(0,0)\}$. 
Throughout of this paper, we fix such a domain $\Delta$. 

Since $\alpha^2\beta^2\ne 4(\alpha^2+\beta^2)$ for every $(\alp,\beta) \in \Delta$, 
and since $\Delta$ is connected and simply connected, one can take a univalent branch of the square root of 
$\alpha^2\beta^2- 4(\alpha^2+\beta^2)$ on $\Delta$. 
We take the  branch such that the value for $(\alp,2) \in \Delta$ is equal to  $-4i$. 
Then we obtain the univalent branch 
of 
\begin{eqnarray}
\gamma=\gamma(\alp,\beta)=\frac{1}{2}\left( \alpha \beta + \sqrt{\alpha^2\beta^2-4(\alpha^2+\beta^2)}\right) 
\end{eqnarray}
on $\Delta$, and hence the univalent branch $\theta: \Delta \to \Xi$ of $\Pi^{-1}$ on $\Delta$. 
\begin{lem}\label{branch}
The map 
$
\pi   \circ \wt \Tr^{-1} \circ \theta:\Delta \to R(N,P)
$
is a homeomorphism onto its image. 
\end{lem}
\begin{proof}
We only need to show that the orbit of $ \theta(\Delta)$ under the action of $\Z_2 \times \Z_2$  on $\Xi$
are mutually disjoint. 
Take two points $(\alp,\beta), (\alp',\beta') \in \Delta$.  
Suppose for contradiction that 
$(\alp,\beta, \gamma(\alp,\beta)), (\alp',\beta',\gamma(\alp',\beta')) \in\Xi $ are equivalent 
under the action of non-trivial element of the covering transformation group $\Z_2 \times \Z_2$. 
Since $\re\,\beta>0$ and $\re\,\beta'>0$, 
one see that $(\alp',\beta',\gamma(\alp',\beta'))=(-\alp,\beta, -\gamma(\alp,\beta))$. 
Then from (3.1) we have 
\begin{eqnarray*}
\gamma(\alp',\beta')&=&\frac{1}{2}\left( \alpha' \beta' + \sqrt{\alpha'^2\beta'^2-4(\alpha'^2+\beta'^2)}\right) \\
&=&\frac{1}{2}\left( -\alpha \beta + \sqrt{\alpha^2\beta^2-4(\alpha^2+\beta^2)}\right). 
\end{eqnarray*}
But this with $\gamma(\alp',\beta')=-\gamma(\alp,\beta)$ implies $\sqrt{\alpha^2\beta^2-4(\alpha^2+\beta^2)}=0$, 
which contradicts to $(\alp,\beta) \in \Delta$. 
\end{proof}

Now let 
$$\Omega:=\pi \circ \wt \Tr^{-1} \circ \theta(\Delta)$$ and 
$$
\Tr:=\left(\pi \circ \wt \Tr^{-1} \circ \theta\right)^{-1}:\Omega \to \Delta. 
$$
Then we obtain the following commutative diagram: 
\begin{eqnarray}
\xymatrix{ 
\wt R(N,P)  \ar[d]^\pi \ar[r]^{\quad \wt \Tr} &  \Xi 
 \\
 \quad \quad R(N,P) \supset \Omega \ar[r]^{\quad \quad \quad \Tr} & \Delta.  \ar[u]_\theta  
 }
\end{eqnarray}
To show that this $\Omega$ and $\Tr$ satisfy the desired property in Lemma \ref{trace}, 
we only need to show that $\Tr(\rho_\alp)=(\alp,2)$ for every $\alp \in \C$. 
This can be seen from the following two facts:  
(i) If we regard $\rho_\alp=\psi(\alp)$ as an element of $\wt R(N,P)$, 
we have $\wt \Tr(\rho_\alp)=(\alp,2,\alp-2i)$. 
(ii) From our choice of the branch $\theta$, we have $\theta(\alp,2)=(\alp,2 , \alp-2i)$.  
Thus we complete the proof of Lemma \ref{trace}.


\section{Bromberg's coordinates for $AH(N,P)$}

This section is devoted to explain the theory of Bromberg in \cite{Br}, 
which tells us the topology of $AH(N,P)$ near the Maskit slice $AH(N,P')$. 
In fact, Bromberg construct a subset of $\C \times \hat \C$ such that 
$AH(N,P)$ is locally homeomorphic to this set 
at every point in $MP(N,P')$. 

\subsection{The Maskit slice}

Given $\mu \in \C$,  we define a representation $\sigma_\mu  \in R(N,P')$ by 
\begin{eqnarray*}
 \sigma_{\mu}(a):=\left(\begin{array}{cc}
 -i\mu & -i \\ 
-i & 
0 \end{array}\right), \quad 
\sigma_{\mu}(b):=\left(\begin{array}{cc}
  1& 2 \\ 
0 & 1
 \end{array}\right). 
\end{eqnarray*}
This representation $\sigma_\mu$ is nothing but the representation $\rho_\alp$ with $\alp=-i\mu$,  
which is defined in the previous section. 
The subset 
$$
\M:=\{\mu \in \C : \sigma_\mu \in AH(N,P')\} 
$$
of $\C$ is also called the {\it Maskit slice}. 
Since  that the map 
$\C \to R(N,P')$ defined by $\mu \mapsto \sigma_\mu$ is a homeomorphism from Lemma \ref{trace2}, 
$\M$ is homeomorphic to $AH(N,P')$, and 
the interior $\inte(\M)$ of $\M$ is homeomorphic to $MP(N,P')$.  
Since $\mu \in \M$ if and only if $-i\mu \in \LL(2)$, we have  
$$
\LL(2)=i\M=\{i \mu : \mu \in \M\}.
$$
Note that  $\M$ is invariant under the translation $\mu \mapsto \mu+2$. 
We refer the reader to \cite{KS} for basic properties of $\M$. 
It is known by Minsky (Theorem B in \cite{Mi}) that $\M$
has two connected components $\M^+, \,\M^-$, 
where $\M^+$ contained in the upper half-plane and $\M^-$ is the complex conjugation of $\M^+$

\subsection{Coordinates for $AH(\hN,\hP)$}

We now introduce a coordinate system on the space $AH(\hN,\hP)$. 
Recall that  
$\hat N$ is $N$ minus a regular tubular neighborhood $W$ of $b \times \{1/2\}$,  and 
$\hat P$ is a union of $P$ and  $T=\bd W$. 
Bromberg's idea in \cite{Br} is that the space $AH(\hN,\hP)$ can be used as a local model of $AH(N,P)$ 
near a point of $AH(N,P')$. 

The fundamental group of $\hN$ is expressed as 
$$
\pi_1(\hat N)=\la a,b,c : [b,c]=\id\ra, 
$$
where $a,b$ is the pair of generators of the fundamental group of $S \times \{0\}  \subset  \hN$, 
and $c$ is freely homotopic to an essential simple closed curve on $T$ that bounds a disk in $W$. 
We regard $\pi_1(T)=\la b,c \ra$. 
The space $R(\hat N,\hat P)$ of representations for $(\hN,\hP)$ is expressed as 
$$
R(\hat N,\hat P)=\{\rho:\pi_1(\hat N) \to \psl  : \tr \rho([a,b])=-2, \, \tr^2 \rho(c)=4 \}/\psl. 
$$
For a given $(\mu,\zeta) \in \C^2$, we define a representation 
$\hat\sigma_{\mu,\zeta} \in R(\hN,\hP)$ by 
\begin{eqnarray*}
\hat \sigma_{\mu,\zeta}(a):=\sigma_\mu(a), \quad 
\hat \sigma_{\mu,\zeta}(b):=\sigma_\mu(b), \quad 
\hat\sigma_{\mu,\zeta}(c):=\left(\begin{array}{cc}
  1& \zeta \\ 
0 & 1
 \end{array}\right). 
\end{eqnarray*}
Then  we have the following: 
\begin{lem}[Lemma 4.5 in \cite{Br}]
The map $\C^2 \to R(\hN,\hP)$ defined by $(\mu,\zeta) \mapsto \hat \sigma_{\mu,\zeta}$
is a homeomorphism. 
\end{lem}
\begin{rem}
Following the rule of notation in \cite{Br}, the representation $\hat \sigma_{\mu,\ze}$ 
should be written as $\sigma_{\mu,\ze}$. But we reserve the notation $\sigma_{\mu,\ze}$ 
for another representation, which will be defined in the next subsection. 
\end{rem}
We define a subset $\B$ of $\C^2$ by 
\begin{eqnarray*}
\B:=\{(\mu,\zeta) \in \C^2 : \hat \sigma_{\mu,\zeta}  \in  AH(\hN,\hP)\}.  
\end{eqnarray*}
Then, by the above lemma, 
 the map 
 $$\B \to AH(\hN,\hP)$$ 
 defined by 
$(\mu,\ze) \mapsto \hat\sigma_{\mu,\ze}$ is a homeomorphism. 
Note that 
 $(\mu,\ze) \in \B$ implies $\mu \in \M$ since 
 the restriction of $\hat \sigma_{\mu,\ze}$ 
 to the subgroup ${\la a, b\ra}$ of $\pi_1(\hN)$ 
 is equal to $\sigma_\mu$. 
 Note also that if $\im\,\ze=0$ then $(\mu,\ze) \not\in\B$; in fact, 
 if $\im\,\ze=0$, it violates discreteness or faithfulness of the representation $\hat \sigma_{\mu,\ze}$.

For any $(\mu,\ze)\in \B$,  the quotient manifold 
$\hat M=\HH^3/\hat  \sigma_{\mu,\ze}(\pi_1(\hN))$ is homeomorphic to the interior of $\hN$, 
and has a rank-2 cusp whose monodromy group is the rank-2 parabolic subgroup 
of $\psl$ generated by 
$\hat  \sigma_{\mu,\ze}(b)$ and $\hat \sigma_{\mu,\ze}(c)$. 
Since 
$$
\hat \sigma_{\mu,\ze}(c^{-k}a)=
\left(\begin{array}{cc}
  -i(\mu-k\ze)& -i \\ 
-i & 0
 \end{array}\right) \quad \text{and} \quad 
\hat  \sigma_{\mu,\ze}(b)=
\left(\begin{array}{cc}
1& 2 \\ 
0 & 1
 \end{array}\right), 
$$
one can see that if  $(\mu,\ze) \in \B$ then $\mu -k\ze \in \M$ for every $k \in \Z$. 
Bromberg showed that the converse is also true if $\im\,\ze \ne 0$ (see Proposition 4.7 in \cite{Br}): 
\begin{thm}[Bromberg]\label{BromLem}
Let $(\mu,\ze) \in \C^2$ with $\im\,\ze \ne 0$. Then $(\mu,\zeta) \in \B$ if and only if 
$\mu -k\ze \in \M$ for every integer $k$. 
\end{thm}

\subsection{Bromberg's coordinates for $AH(N,P)$}

Following \cite{Br}, 
we now introduce a coordinate system on $AH(N,P)$ by using the coordinate system on $AH(\hN,\hP)$ 
introduced in the previous subsection. 

Now let 
$$
\B^+:=\{(\mu,\ze) \in \B: \im\,\ze>0\}
$$
 and define a set 
$\A \subset \C \times \hC$ by 
$$
\A:=\B^+ \cup (\M \times \{\fty\}).
$$
The following theorem due to Bromberg claim that the set $\A$ 
can be used for a local model of $AH(N,P)$ at every point of $MP(N,P') \subset AH(N,P)$. 
 \begin{thm}[Bromberg (Theorem 4.13 in \cite{Br})]\label{BromThm}
 For any $\nu \in \inte(\M)$,  
 there exist a neighborhood $U$ of $(\nu,\fty)$ in $\A$, 
 a neighborhood $V$ of $\sigma_\nu$ in $AH(N,P)$,  and a homeomorphism 
 $
 \Phi:U \to V.
 $
\end{thm} 
\begin{rem}
Although Bromberg restricted to the case that $\nu \in \inte(\M^+)$ in \cite{Br},  
it is obvious that the same argument works well for $\nu \in \inte(\M^-)$. 

\end{rem}

In this situation, we say that $(\mu,\ze) \in U$ is  {\it Bromberg's coordinates} of 
the representation 
$\Phi(\mu,\ze) \in V$. 
In what follows, we also write  
$$
\sigma_{\mu,\ze}:=\Phi(\mu,\ze). 
$$

We now briefly explain the definition of the map $\Phi:U \to V, \ (\mu,\ze) \mapsto \sigma_{\mu,\ze}$ 
to what extent we need in the following argument. 
See \cite{Br} for the full details.  
Given $\nu \in \inte(\M)$, a neighborhood $U$ of  $(\nu,\fty)$ in $\A$ 
is chosen sufficiently small so that the following argument works well.  
Let $(\mu,\ze) \in U$. 
If $\ze=\fty$ then $\sigma_{\mu,\fty}
$ is defined to be $\sigma_\mu$. 
If $\ze \ne \fty$, the quotient manifold 
$$
\hat M_{\mu,\ze}=\HH^3/\hat \sigma_{\mu,\ze}(\pi_1(\hN)) 
$$ 
has a rank-$2$ cusp whose monodromy group is  
generated by $\hat \sigma_\mz(b)$ and $\hat \sigma_\mz(c)$.   
Since we are choosing $U$ sufficiently small, it follows from 
the filling theorem due to Hodgson, Kerckhoff and Bromberg (see Theorem 2.5 in \cite{Br}) that 
there exists a  {\it $c$-filling} $M_\mz$ of $\hat M_\mz$ 
for every $(\mu,\ze) \in U$ with $\ze \ne \fty$. 
More precisely,  
there is a complete hyperbolic manifold $M_\mz$ homeomorphic to the interior of $N$ 
and an embedding 
$$
\phi_{\mu,\ze}:\hat M_{\mu,\ze} \to M_{\mu,\ze}
$$
which satisfy the following properties: 
\begin{enumerate}
\item  
the image of $\phi_{\mu,\ze}$ is equals to $M_{\mu,\ze}$ minus 
the geodesic representative of 
$(\phi_{\mu,\ze})_*(\hat \sigma_\mz(b))$, 
\item 
$(\phi_{\mu,\ze})_*(\hat \sigma_\mz(c))$ is trivial in $\pi_1(M_\mz)$, 
and  
\item 
$\phi_\mz$ extends to a conformal map between the conformal boundaries 
of $\hat M_\mz$ and $M_\mz$. 
\end{enumerate}
The map $\phi_{\mu,\ze}$ is called the {\it $c$-filling map}. 
We will define $\sigma_{\mu,\ze}$ to be an element in $AH(N,P)$ associated to $M_{\mu,\ze}$. 
To this end, we need to determine a marking $N \to M_{\mu,\ze}$. 
Since the restriction of the representation $\hat \sigma_{\mu,\ze}$ 
to the subgroup $\la a,b\ra \subset \pi_1(\hN)$ is equal to $\sigma_\mu$,  
the manifold $M_\mu=\HH^3/\sigma_\mu(\pi_1(N))$ covers $\hat M_{\mu,\ze}$. 
The covering map is denoted by 
$$\Pi_{\mu,\ze}:M_\mu \to \hat M_{\mu,\ze}.
$$ 
Let $f_\mu:N \to M_\mu$ be a homotopy equivalence which induces $\sigma_\mu$. 
Then $\sigma_{\mu,\ze}$ 
is defined to be a representation of $\pi_1(N)$ into $\psl$ 
induced form $\phi_{\mu,\ze} \circ \Pi_{\mu,\ze} \circ f_\mu$; 
\begin{eqnarray*}
\begin{CD}
N@>{f_\mu}>{\text{marking}}> M_\mu @>{\Pi_{\mu,\ze}}>{\text{covering}}>
 \hat M_{\mu,\ze} @>{\phi_{\mu,\ze}}>{\text{filling}}> M_{\mu,\ze}. 
\end{CD}
\end{eqnarray*}
This $\sigma_{\mu,\ze}$ is faithful, and hence, is contained in $AH(N,P)$ (see Lemma 3.6 in \cite{Br}).  
Note from the construction of $\sigma_{\mu,\ze}$ that the geodesic in $M_{\mu,\ze}$ 
associated to $\sigma_{\mu,\ze}(a)$ is homotopic to 
the image of the geodesic $\hat M_{\mu,\ze}$ associated to 
$\hat\sigma_{\mu,\ze}(a)$ by $\phi_{\mu,\ze}$.

\section{Relation between the trace coordinates and Bromberg's coordinates}

Let us consider the situation in Theorem \ref{BromThm}.  
We may assume that $V$ is contained in the domain $\Omega$ of the trace coordinate map.  
In this section, we will study the relation between 
Bromberg's coordinates $(\mu,\ze) \in U$ of $\sigma_{\mu,\ze} \in V$ 
and its trace coordinates $(\alp,\beta)=\Tr(\sigma_{\mu,\ze})$. 
More precisely, we will observe in Theorem \ref{estimate} that  
$(\mu, \ze)$ is approximated by $(i\alp, 4\pi i/\lam(\beta))$, 
where $\lam(\beta)=2 \cosh^{-1}(\beta/2)$ 
is the complex length of $\sigma_{\mu,\ze}(b)$.

\subsection{Complex length}

For any  element $g \in \psl$, its {\it complex length} $l(g) \in \C$  is a value which satisfies 
$$
{\tr^2 g}= 4\cosh^2
\left(\frac{l(g)}{2}\right).  
$$
If $g$ is not parabolic, this is equivalent to say that $g$ is conjugate to the M\"{o}bius transformation 
$z \mapsto e^{l(g)}z$. 
For a loxodromic element $g \in \psl$, 
its complex length $l(g)$ determined uniquely 
 if we take it in the set 
$$
\Lam:=\{z \in \C : \re \,z > 0, \, -\pi< \im \,z \le \pi\}. 
$$
In what follows, we always assume that $\l(g) \in \Lam$ for loxodromic transformation $g$.  

We now want to fix one-to-one correspondence  between 
the complex length $l(g)$ of loxodromic element $g \in \ssl$ and its trace $\tr\,g$. 
Note that the map $z \mapsto 2 \cosh(z/2)$ takes the interior of $\Lambda$ 
into the right-half plane  
$$\C_+:=\{z \in \C : \re\,z>0\}.$$
We define a map 
$$
\lam:\C_+ \sm (0,2) \to \Lambda
$$
as its inverse. 
Then we have 
\begin{eqnarray*}
\lam(z) &=&2\cosh^{-1}\left(\frac{z}{2}\right) \\
&=& 2 (z-2)^{1/2}+o(z-2) \quad (z \to 2), 
\end{eqnarray*}
where the real part of a square root is chosen positive. 
We have 
$$\lam(\tr\, g)=l(g)$$ 
for every loxodromic element $g \in \ssl$ with $\tr \,g \in \C_+ \sm (0,2]$.

\subsection{Main estimates}

The following theorem tells us a relation between Bromberg's coordinates and the trace coordinates 
for representations close to the Maskit slice. 

\begin{thm}\label{estimate}
Let $\nu \in \inte(\M)$.  For any $\ep>0$, 
we can choose a neighborhoods $U$ of $(\nu,\fty)$  and 
$V$ of $\sigma_\nu$ in Theorem \ref{BromThm} 
so that they also  satisfy the following: 
$V$ is contained in the domain $\Omega$ of the trace coordinate map, and 
 for any $(\mu,\ze) \in U$ with $\ze  \ne \fty$, we have 
\begin{enumerate}
\item $|\mu-i\alp| \le \ep$, and 
\item $|\ze-4\pi i/\lam(\beta)| \le \ep \, \im\,\ze$, 
\end{enumerate}
where $(\alp,\beta)=\Tr(\sigma_{\mu,\ze})$ is the trace coordinates of $\sigma_{\mu,\ze}$
\end{thm}

\begin{rem}
These estimates 1 and 2 follow from the fact that 
we can choose the $c$-filling map 
$\phi_{\mu,\ze}:\hat M_{\mu,\ze} \to M_{\mu,\ze}$ close to the isometry 
outside a neighborhood of the rank-$2$ cusp.    
Then the estimates 1 and 2 are obtained from estimates
 due to McMullen (Lemma 3.20 in \cite{Mc1}) and 
  Magid (Theorem 1.2 in \cite{Mag}), respectively.  
\end{rem}

\begin{proof} [Proof of Theorem \ref{estimate}] 
Let us take a neighborhood $U$ of  $(\nu,\fty)$ in $\A$, 
a neighborhood $V$ of $\sigma_{\nu}$ in $AH(N,P)$ 
and a homeomorphism $\Phi:U \to V$ as in the statement of Theorem \ref{BromThm}. 
We may assume that $V \subset \Omega$. 
We will show below that estimates 1 and 2 are obtained 
 if we modify $U$ sufficiently small. 
 
For $(\mu,\ze) \in U$ with $\ze \ne \fty$, 
let 
$$
\phi_{\mu,\ze}:\hat M_{\mu,\ze} \to M_{\mu,\ze}
$$ 
be the $c$-filling map.   
To control the distortion of the map $\phi_{\mu,\ze}$, we need to recall the notion of normalized length. 

Suppose that  $\delta>0$ is less than the Margulis constant for hyperbolic 3-manifolds, and let 
$\mathbb{T}_\delta(T)$ denote the component of $\delta$-thin part of $\hat M_{\mu,\ze}$ 
associated to the rank-2 cusp.  
We endow the boundary $\bd \mathbb{T}_\delta(T)$  of $\mathbb{T}_\delta(T)$ 
with the natural  Euclidean metric. 
The marking map $\hat N \to \hat M_{\mu,\ze}$ 
induces a marking map $T \to \bd \mathbb{T}_\delta(T)$. 
Via this marking, the pair of generators $b,c $ of $\pi_1(T)$ are also regarded as the pair of generators of 
 $\pi_1(\bd \mathbb{T}_\delta(T))$. 
In this setting, the {\it normalized length} $L(c)$ 
of the free homotopy class of $c \subset \bd \mathbb{T}_\delta (T)$  
is defined by 
$$
L(c):=\frac{\mathrm{length}(c')}
{\sqrt{\mathrm{Area}(\bd \mathbb{T}_\delta(T))}}, 
$$
where $\mathrm{length}(c')$ is the Euclidean length of the geodesic representative $c'$ of $c$ in 
$\bd \mathbb{T}_\delta (T)$. 
This $L(c)$ does not depend on the choice of $\delta$. 
Since 
$$
\hat \sigma_{\mu,\ze}(b)=
\left(\begin{array}{cc}
 1 & 2 \\ 
 0& 1
 \end{array}\right) 
 \quad \text{and} \quad 
 \hat \sigma_{\mu,\ze}(c)=\left(\begin{array}{cc}
 1 & \ze \\ 
 0& 1
 \end{array}\right), 
$$
the normalized length $L(c)$ can be calculated concretely as 
$$
L(c)=\frac{|\zeta|}{\sqrt{2 \, \im \,\zeta}}. 
$$

 For any given $K>0$, we can choose the neighborhood $U$ sufficiently small 
 so that  for all $(\mu,\ze) \in U$ with $\ze \ne \fty$,  
 the normalized length $L(c)$ of $c$ at the rank-$2$ cusp of $\hat M_{\mu,\ze}$ is greater than $K$. 
We will show below that if we take such $K$ sufficiently large, the estimates 1 and 2 hold.

We may assume that there is a uniform upper bound of $|\mu|$  for  $(\mu,\ze) \in U$. 
Then, since $\tr^2 \hat \sigma_{\mu,\ze}(a)=-\mu^2$, there is an upper bound 
$R>0$ for hyperbolic lengths of 
geodesic representatives $a^*$ of $\hat \sigma_{\mu,\ze}(a)$ in 
$\hat M_{\mu,\ze}$ for all $(\mu,\ze) \in U$ with $\ze \ne \fty$. 
Therefore we can take $\delta>0$ small enough 
so that  the unit neighborhood ${\cal N}(a^*,1)$ of 
$a^*$ in $\hat M_{\mu,\ze}$ does not intersect the 
$\delta$-thin part $\mathbb{T}_\delta(T) \subset \hat M_{\mu,\ze}$ 
for all $(\mu,\ze) \in U$ with $\ze \ne \fty$. 

It follows from the filling theorem due to Hodgson, Kerckhoff and Bromberg (see Theorem 2.5 in \cite{Br}) that 
for $\delta>0$ chosen as above and for any $\ep_1>0$, there exists $K>0$ such that 
if the normalized length  of $c$  in $\hat M_{\mu,\ze}$ is greater than $K$ 
then the $c$-filling map can be chosen  
so that it restricts to a $(1+\ep_1)$-bi-Lipschitz diffeomorphism 
$$
\phi_{\mu,\ze}: \hat M_{\mu,\ze} \sm \mathbb{T}_\delta(T) 
\to M_{\mu,\ze} \sm \mathbb{T}_\delta(b^*),  
$$
where $\mathbb{T}_\delta(b^*) \subset M_{\mu,\ze}$ is the $\delta$-Margulis tube 
of the geodesic representative $b^*$ of $\sigma_{\mu,\ze}(b)$; i.e., 
the core curve of the filled torus in $M_{\mu,\ze}$. 
We can now apply a theorem of McMullen (Lemma 3.20 in \cite{Mc1}) to obtain 
$$
|\tr^2 (\hat \sigma_{\mu,\ze}(a))-\tr^2 ( \sigma_{\mu,\ze}(a))|<C(R)\ep_1, 
$$
where $C(R)>0$ is a constant which depends only on $R$. 
(Recall that $R$ is the upper bounds of the hyperbolic length of $a^*$.)   
Since $\tr^2\, \hat \sigma_{\mu,\ze}(a)=-\mu^2$, $\tr^2\, \sigma_{\mu,\ze}(a)=\alpha^2$,  
and since $\alp$ is close to $-i\mu$,   
we obtain 
$$
|\mu-i\alp|<\ep
$$
for a  given $\ep>0$ by taking  $K$ large enough. 
Thus we obtain the first estimate. 

We next show the second estimate. 
One can expect to obtain this kind of estimate since 
the Teichm\"{u}ller parameter of the torus 
$\bd  \mathbb{T}_\delta(T) \subset \hat M_{\mu,\ze}$  
with respect to the generators $b,c$ is equal to 
$\ze/2$ and that of  $\bd \mathbb{T}_\delta(b^*) \subset M_{\mu,\ze}$ is equal to 
$2\pi i/\lam(\beta)$,  
and since there is a bi-Lipschitz map of small distortion between these tori. 
Magid accomplish this estimate in \cite{Mag}.  
In fact, by simplifying his estimates (ii) and (iv) of Theorem 1.2 in \cite{Mag}, 
we see that  there is some constant $C>0$ such that  if the normalized length 
$L(c)=|\zeta|/\sqrt{2 \, \im \,\zeta}$ of $c$ is sufficiently large, we have 
\begin{eqnarray}
\left|\lam(\beta)-\frac{4\pi i}{\ze}\right|\le C \frac{(\im\,\ze)^2}{|\ze|^4}=\frac{4C}{L(c)^4}. 
\end{eqnarray} 
One can also see from $L(c)=|\zeta|/\sqrt{2 \, \im \,\zeta}$ that  $\re(4\pi i/\ze)=2\pi/L(c)^2$. 
Combining this with (5.1), we have 
\begin{eqnarray}
|\lam(\beta)|>\frac{1}{2}\left|\frac{4\pi i}{\ze}\right|=\frac{2\pi}{|\ze|}
\end{eqnarray}
for $L(c)$  large enough.  
Finally, multiplying $|\ze/\lam(\beta)|$ on both sides of (5.1) and using the estimate (5.2),  
we obtain 
$$
\left|\ze-\frac{4\pi i}{\lam(\beta)}\right| \le   C' \frac{(\im \ze)^2}{|\ze|^2}=\frac{C'}{2 L(c)^2} \im\,\ze<\ep \, \im\,\ze 
$$
for a given $\ep>0$ 
if $L(c)$ is large enough.  
Thus we obtain the second estimate. 
\end{proof}

\subsection{$\A$ and $\D$} 

Since the shape of $\A$ is well understood from Theorem \ref{BromLem}, 
we can expect to understand the shape of $\D_\tr$ from that of $\A$. 
To apply Theorem \ref{estimate},  it is convenient to consider the image of the set $\D_\tr$ 
by the transformation $(\alp,\beta) \mapsto (i\alp,4\pi i/\lam(\beta))$.  
More precisely, we define a map 
$$
F:\C \times (\C_+ \sm (0,2))\to \C \times \hC
$$ 
by 
$$
F(z,w):=\left(iz, \frac{4\pi i}{\lam(w)}\right) 
$$
and set
$$
\D_\tr^+:=\{(\alp,\beta) \in \D_\tr : \beta \in \C_+\}
$$
and 
$$
\D:=F(\D_\tr^+).
$$
Note that, since we are interested in the shape of $\D_\tr$ where  the second entry is close to $2$, 
we may restrict our attention to $\D_\tr^+$. 
Note also that $F(z,2)=(iz, \fty)$ for every $z$. 

Let us now consider the situation of Theorem \ref{estimate}. 
We define a homeomorphism $\varphi$ from $U$ onto  its domain by 
$\varphi=F \circ   \Tr \circ \Phi$; 
\begin{eqnarray*}
\begin{CD}
\varphi:U@>{\Phi}>> V  @>{\Tr}>> \D_\tr @>{F}>> \D. 
\end{CD}
\end{eqnarray*}
Then by definition we have  $\varphi(\mu,\fty)=(\mu,\fty)$ for any $(\mu,\fty) \in U$. 
It follows from Theorem \ref{estimate} 
the point $\varphi(\mu,\ze)$ is close to $(\mu,\ze) \in U$ even if $\ze \ne \fty$. 
Therefore,  we expect that the shape of $\A$ is similar to  that of $\D$ in a neighborhood of $(\nu,\fty)$ 
for every $\nu \in \inte(\M)$. 

We will justify this expectation  in Propositions \ref{ball1} and \ref{ball2} below.   
In what follows, we denote by $B_\ep(z)$ the $\ep$-neighborhood of $z$ in $\C$, and 
 by  $B_\ep(z,w)$ the $\ep$-neighborhood of $(z,w)$ in $\C^2$.

\begin{prop}\label{ball1}
For any $\nu \in \inte(\M)$, there exists $\ep_0>0$ which satisfy the following: 
For any $0<\ep<\ep_0$ and $I>0$, there exists $K>0$ such that 
for all $z \in \C$ with $|z|>K$ and $0<\im\,z<I$, 
$
B_\ep(\nu,z) \subset \A
$
implies
$
B_{\ep/2}(\nu,z) \subset \D. 
$
\end{prop}

\begin{proof}
For any fixed $\nu \in \inte(\M)$, let us take 
 neighborhoods $U \subset \A$, $W \subset \D$ 
 of $(\nu,\fty)$ such that $\vphi=F \circ \Tr \circ \Phi$ is a homeomorphism from $U$ onto $W$. 
We may assume that $U$ is of the form  
$$
U=\A \cap \{(\mu,\ze) \in \C^2 : |\mu-\nu|<\ep_0, |\ze|>K/2\}
$$
for some $\ep_0>0$ and $K>0$. 
Let us take $0< \ep<\ep_0$ and $I>0$ arbitrarily.   
One can see from Theorem \ref{estimate} and its proof that  if we choose $K$ large enough, 
we may also assume that 
\begin{eqnarray}
d_{\C^2}(\vphi(\mu,\ze),(\mu,\ze))<\frac{\ep}{8}
\end{eqnarray}
holds for every $(\mu,\ze) \in U$ with  $0<\im\,\ze<2I$. 
(Note that if $|\ze| \to \fty$ then $|\ze|/\sqrt{2 \im\,\ze} \to \fty$.)

Now let us take $z \in \C$ with $|z|>K$ and $0<\im\,z<I$, 
and suppose that $B_\ep(\nu,z) \subset \A$.  
Then  $B_\ep(\nu,z) \subset U$ and $0<\im\,\ze<2I$ for every 
$(\mu,\ze) \in B_\ep(\nu,z)$  
(since $K$ and $I$ are larger than $\ep$). 
Thus the inequality (5.3) holds for every $(\mu,\ze) \in B_{\ep}(\nu,z)$. 
Using this fact, we will show that 
$$
B_{\ep/2}(\nu,z) \subset \varphi(B_\ep(\nu,z)),  
$$
which implies that $B_{\ep/2}(\nu,z) \subset \D$. 

Suppose for contradiction that there exists some 
$p \in B_{\ep/2}(\nu,z) \sm  \varphi(B_\ep(\nu,z))$. 
Let consider a line segment 
$$
\gamma(t):=(1-t)\varphi(p) +t p, \quad t \in [0,1]
$$
in $\C^2$ which joins $\varphi(p)$ to $p$. 
Since $d_{\C^2}(\vphi(p),p)<\ep/8$ and $p \in B_{\ep/2}(\nu,z)$,  we have 
$\gamma([0,1]) \subset B_{5\ep/8}(\nu,z)$.  
Now let 
$$
t_\fty:=\inf\{t : \gamma(t)  \not\in \varphi(B_\ep(\nu,z))\}. 
$$
Since $\varphi(p)$ lies in $\varphi(B_\ep(\nu,z))$ but $p$ does not,  
and since $\varphi(B_\ep(\nu,z))$ is open, one see that $0<t_\fty \le 1$. 
Let take an increasing sequence $t_n \to t_\fty \, (n \to \fty)$ 
and let $q_n:=\varphi^{-1}(\gamma(t_n)) \in B_\ep(\nu,z)$. 
Since $\varphi(q_n)(=\gamma(t_n))$ lie in $B_{5\ep/8}(\nu,z)$ and 
$d_{\C^2}(\varphi(q_n),q_n)<\ep/8$, we have 
$q_n \in B_{3\ep/4}(\nu,z)$ for all $n$. 
Therefore an accumulation point $q_\fty$ of $\{q_n\}$ 
lies in $B_{\ep}(\nu,z)$. 
It follows form the continuity of $\varphi$ that $\varphi(q_\fty)=\gamma(t_\fty)$. 
Since $\varphi$ is local homeomorphism at $q_\fty$, 
this contradicts the definition of $t_\fty$. 
Thus we obtain $B_{\ep/2}(\nu,z) \subset \varphi(B_\ep(\nu,z)) \subset \D$. 
 \end{proof}

We set 
$$
B_{\ep,I}(\nu):=B_\ep(\nu) \times \{ z \in \C: \im \,z>I\}. 
$$

\begin{prop}\label{ball2}
For any $\nu \in \inte(\M)$, there exists $\ep_0>0$ which satisfy the following: 
For any $0<\ep<\ep_0$ there is $I>0$ such that 
$
B_{\ep,I}(\nu) \subset \A
$
implies 
$
 B_{\ep/2,2I}(\nu)  \subset \D. 
$
\end{prop}

\begin{proof}
The proof is almost parallel to that of Proposition \ref{ball1}. 
For any fixed $\nu \in \inte(\M)$, 
let us consider the homeomorphism $\varphi:U \to W$ as in the proof of Proposition \ref{ball1}. 
One can see from Theorem \ref{BromLem} that 
there exist $\ep_0>0$ and $I>0$ such that $B_{\ep_0,I}(\nu) \subset U$. 
Now let us take $0<\ep<\ep_0$ arbitrarily. 
Theorem \ref{estimate} implies that if we choose $I$ large enough, 
we have the following: 
for any $(\mu,\ze) \in B_{\ep,I}(\nu)$,  $(\mu',\ze'):=\varphi(\mu,\ze)$ satisfies 
$|\mu'-\mu|<\ep/8$ and $|\ze'-\ze|< (\ep/8) \,\im\,\ze$. 
Using this fact, we can show  that 
$$
B_{\ep/2,2I}(\nu) \subset \varphi(B_{\ep,I}(\nu)),  
$$
which implies that $B_{\ep/2,2I}(\nu) \subset \D$. 
The remaining  argument is almost the same to that of Proposition \ref{ball1}, so we leave it for the reader. 
\end{proof}

\section{Main Results}

In this section, we will show our main results, Theorems \ref{tan} and \ref{horo}. 
More precisely, 
for a given sequence $\beta_n \in \C \sm [-2,2]$ converging to $2$, 
we consider the Hausdorff limit of the linear slices $\LL(\beta_n)$ and the 
Carath\'{e}odory limit of the interiors $\inte(\LL(\beta_n))$ of the linear slices. 

\subsection{Horizontal slices of $\A$}

We first consider horizontal slices of $\A$, which will appear as limits of linear slices. 
Let  $\M(\zeta)$ denote the slice of $\A$ by fixing the second entry 
$\zeta \in \C \cup \{\fty\}$ 
in the product structure; that is, 
$$
\M(\ze):=\{\mu \in \C : (\mu,\zeta) \in \A\}. 
$$
By definition of $\A$, one see that (i) $\M(\ze)$ lies in $\M$ for every $\ze$, 
(ii) $\M(\ze)$ is empty  if $\im\,\ze \le 0$, and that  (iii) $\M(\fty)=\M$. 
It follows from Theorem \ref{BromLem} that  if  $\im\,\ze>0$ the set $\M(\ze)$ can be written as 
\begin{eqnarray}
\M(\zeta)= \bigcap_{k \in \Z}(k\zeta+\M),
\end{eqnarray}
where $k\ze+\M=\{k\ze+\mu : \mu \in \M\}$. 
(Note that (6.1) does not hold if $\im\,\ze \le 0$.) 
Note that  $\M(\ze)$ is invariant under the action of $\la z+2,z+\ze \ra$. 
It is known by Wright \cite{Wr} that the stripe $\{z \in \C : -1 \le \im\,z \le 1\}$ does not intersect $\M$. 
Therefore one see that $\M(\ze)=\emptyset$ if $0<\im\,\ze \le2$. 

We now consider relationship between  horizontal slices of $\A$ and linear slices, or horizontal slices of $\D_\tr$. 
By definition, we have 
\begin{eqnarray*}
\alp \in \LL(\beta) \ \iff \  (\alp,\beta) \in \D_\tr \ \iff \   
\left(i\alp,\frac{4\pi i}{\lam(\beta)}\right) \in  \D 
\end{eqnarray*}
and 
\begin{eqnarray*}
   \left(i\alp,\frac{4\pi i}{\lam(\beta)}\right) \in \A \ 
    \iff \ \alpha \in i  \M\left(\frac{4\pi i}{\lam(\beta)}\right). 
\end{eqnarray*}
Recall from Theorem \ref{estimate} that 
$(\mu,\ze) \in \A$ is almost equivalent to  $(\mu,\ze) \in \D$ 
if  $\mu$ lies in $\inte(\M)$ and $|\ze|$ is large enough. 
Therefore we may expect that $\LL(\beta)$ is similar to $i\M({4\pi i/\lam(\beta)})$ 
when $\beta$ is close to $2$.  
We will justify this observation below. 
To this end, we first recall the definitions of Hausdorff convergence and Carath\'{e}odory  convergence. 

\begin{defn}[Hausdorff convergence]
Let 
$F_n \ (n \in \N),\,F_\fty$ be closed subsets in $\C$.  
We say that the sequence $F_n$ converges $F_\fty$ in the sense of Hausdorff 
if the following two conditions are satisfied: 
\begin{enumerate}
\item For any $x_\fty \in F_\fty$, there is a sequence $x_n \in F_n$ 
such that $x_n \to x_\fty$. 
\item If there is a sequence 
$x_{n_j} \in F_{n_j}$ such that $x_{n_j} \to x_\fty$, then  $x_\fty \in F_\fty$. 
\end{enumerate}
\end{defn}

\begin{defn}[Carath\'{e}odory convergence]
Let $\Omega_n \ (n \in \N),\,\Omega_\fty$ be open subsets in $\C$. 
We say that the sequence $\Omega_n$ converges to $\Omega_\fty$ 
in the sense of Carath\'{e}odory if the following two conditions are satisfied: 
\begin{enumerate}
\item For any compact subset $X$ of $\Omega_\fty$, 
$X \subset \Omega_n$ for all large $n$. 
\item If there is an open subset $O$  of $\C$ and an infinite sequence $\{n_j\}_{j=1}^\fty$ 
such that $O \subset \Omega_{n_j}$, then $O \subset \Omega_\fty$. 
\end{enumerate}
\end{defn}

Note that closed subsets $F_n \subset \C$ converge to $F_\fty \subset \C$ in the sense of Hausdorff 
if and only of their complements $\C \sm F_n$  converge to 
$\C \sm F_\fty$ in the sense of Carath\'{e}odory. 

The next lemma implies that  
$\M(\ze_n)$ converge to $\M$ 
if and only if $\im\,\ze_n \to \infty$, which is a direct consequence of (6.1): 

\begin{lem} 
Suppose that a sequence $\{\ze_n\}_{n=1}^\fty$  in $\C$ with $\im\,\ze_n>0$ converges to $\fty$ in $\hC$. 
Then the followings are equivalent: 
\begin{enumerate}
\item $\im\,\ze_n \to \fty$ as $n \to \fty$. 
\item $\M({\ze_n})$converge to $\M$ in the sense of Hausdorff as $n \to \fty$. 
\item $\inte(\M({\ze_n}))$ converge to $\inte(\M)$ in the sense of Carath\'{e}odory as $n \to \fty$. 
\end{enumerate}
\end{lem}

\subsection{Horocyclic and tangential convergence}

To describe our main theorems, we also need the following definition (see Figure 1): 
\begin{defn}\label{horotan}
Suppose that a sequence $\{\lam_n\}_{n=1}^\fty$ in the right-half plane 
$\C_+=\{z \in \C \,|\, \re \,z>0\}$ converges to $0$. 
We say that $\lam_n \to 0$ \it{horocyclically} if for any $\epsilon>0$,  
$|\lam_n -\epsilon|<\epsilon$ for all large $n$, and that
 $\lam_n \to 0$ \it{tangentially} if there is a constant $\epsilon_0>0$ 
such that   
$|\lam_n -\epsilon_0|>\epsilon_0$ for all  $n$.  
\end{defn}
\begin{figure}
\begin{center}
\includegraphics[width=8cm, bb=0 0 768 329]{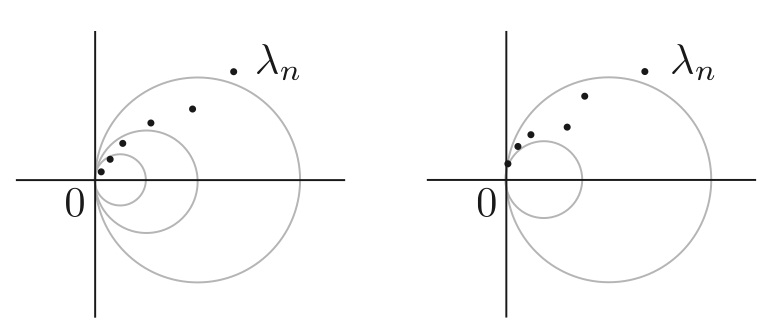}
\caption{Horocyclic convergence (left) and tangential convergence (right). }
\end{center}
\end{figure}
Note that $\lam_n \to 0$ horocyclically if and only if $|\im(2\pi i/\lam_n)| \to \fty$, 
and that tangentially if and only if $|\im(2\pi i/\lam_n)|$ are uniformly bounded above. 

When  a sequence $\beta_n \in \C \sm[-2,2]$ converges to $2$, 
the limit of the sequence $\LL(\beta_n)$ depends on whether  $\lam(\beta_n) \to 0$ horocyclically  
or  tangentially.  
The essence of the difference between  horocyclic and tangential convergence can be found 
in the next theorem on geometric limits of cyclic groups, which was first observed by J{\o}rgensen. 
See, for example, Theorem 3.3 in \cite{It} for the proof. 

We say that a sequence of discrete subgroups $G_n$ of $\psl$ 
converges {\it geometrically} to a subgroup $G$ of $\psl$ if 
$G_n$ converge to $G$ in the sense of Hausdorff as closed subsets of $\psl$.  

\begin{thm}\label{cyclic}
Suppose that a sequence  $B_n$  of 
loxodromic elements converges to 
$$
B=\left(\begin{array}{cc}1 & 2 \\0 & 1\end{array}\right)
$$
in $\psl$. Let $\lam_n$ denote the complex length of $B_n$. 
 Then we have the following: 
\begin{enumerate}
\item 
If  $\lam_n \to 0$ horocyclically, then the sequence $\la B_n \ra$ converges geometrically to $\la B \ra$.  

\item  Suppose that $\lam_n \to 0$ tangentially.  
We further assume  that 
there exists a complex number $\xi$  with $\im\,\xi \ge 0$ 
and a sequence $m_n$ of integers  with $|m_n| \to \fty$ 
such that 
\begin{equation*}
\lim_{n \to \fty} \left(\frac{2\pi i}{\lam_n} -m_n \right)=\xi. 
\end{equation*}
In this situation, we have 
$$
\lim_{n \to \fty}B_n^{-m_n}=C:=\left(\begin{array}{cc}1 & 2\xi \\0 & 1\end{array}\right). 
$$
In addition, if  $\im\,\xi \ne 0$,  
the sequence $\la B_n \ra$ converges  geometrically to the 
rank-$2$ parabolic group $\la B, C \ra$. 
\end{enumerate}
\end{thm}

\begin{rem}
When $\lam_n \to 0$ tangential,  there is a constant $M>0$ such that 
$0<\im(2\pi i/\lam_n)<M$ for every $n$.   
Therefore we may assume that, 
by pass to a subsequence if necessary,  the sequence  
$2\pi i/\lam(\beta_n)$ converges to some $\xi \in \C$ with $\im\,\xi \ge 0$  up to the action of $z \mapsto z+1$.  
\end{rem}

\subsection{Main theorem for tangential convergence}

We can now state our main theorem for linear slices $\LL(\beta_n)$ 
such that  $\lam(\beta_n)$ converge tangentially to $0$. 
See Figure \ref{tr_conv}, left column. 

\begin{thm}\label{tan}
Let 
$\{\beta_n\}_{n=1}^\fty$ be a sequence in $\C \sm [-2,2]$ 
which converges to $2$ as $n \to \fty$. 
Suppose that 
$\lam(\beta_n)$ converge tangentially to $0$. 
We further assume that there exists a complex number $\xi$ with $\im\,\xi \ge 0$
and a sequence $m_n$ of integers  with $|m_n| \to \fty$ 
such that 
$$
\lim_{n \to \fty} \left(\frac{2\pi i}{\lam(\beta_n)}-m_n\right)= \xi. 
$$
Then we have the following: 
\begin{enumerate}
\item $\LL(\beta_n)$ converge to $i \M({2\xi})$ in the sense of Hausdorff as $n \to \fty$. 
\item $\inte(\LL(\beta_n))$ converge to $\inte(i \M({2\xi}))$ in the sense of Carath\'{e}odory as $n \to \fty$. 
\end{enumerate}
\end{thm}

\begin{figure}
\begin{center}
\includegraphics[height=15cm, bb=0 0 675 1015]{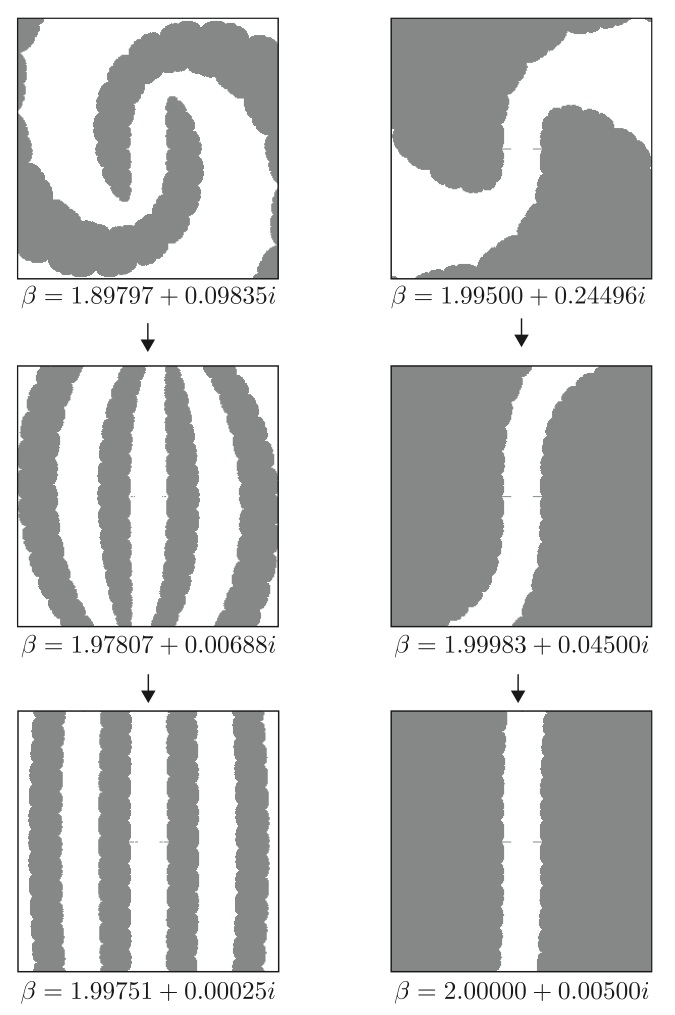}
\caption{Computer-generated figure of linear slices $\LL(\beta)$
 (gray parts) for $\beta$ close to $2$, restricted to the square of width $24$ centered at $0$. 
Left column corresponds to tangential convergence $\lam(\beta) \to0$, where 
$\lam(\beta)$ are points on  the circle $|z-1|=1$ 
whose imaginary part equal $0.7$ (top),  $0.3$ (middle) and $0.1$ (bottom). 
Right column corresponds to horocyclic convergence $\lam(\beta) \to0$, 
where $\lam(\beta)$ equal $0.7+0.7i$ (top),  $0.3+0.3i$ (middle) and $0.1+0.1i$ (bottom).  }
\label{tr_conv}
\end{center}
\end{figure}

The following lemma is an essential part of the proof of Theorem \ref{tan}. 

\begin{lem}\label{tanlem}
Under the same assumption as in Theorem  \ref{tan}, 
we have the following: 
For any $\alp \in \inte(i \M({2\xi}))$ 
there exists $\ep>0$  such that 
$B_\ep(\alp) \subset  \inte(\LL(\beta_n))$
for all large $n$. 
\end{lem}

\begin{proof} 
Suppose that $\alpha  \in \inte(i  \M({2\xi}))$. 
Then  $(i\alp,2\xi) \in \inte(\A)$. 
Let $\ep_0>0$ be the constant in Proposition \ref{ball1} for $\nu=i\alp$. 
Since $(i\alp,2\xi) \in \inte(\A)$, one can find $0<\ep<\ep_0$ such that 
$B_\ep(i \alpha, 2\xi) \subset \A$. 
Since the set $\A$ is invariant under the action $(z,w) \mapsto (z,w+2)$, 
we have $B_\ep(i \alpha, 2\xi+2m_n) \subset \A$. 
Let $K>0$ be the constant in Proposition \ref{ball1} 
for $\ep>0$ chosen above and $I=\im(2 \xi)+1$. 
Since $|m_n| \to \fty$ as $n \to \fty$,   
we have $|2\xi+2m_n|>K$ for all large $n$. 
Then by Proposition \ref{ball1},  we have  
$$
B_{\ep/2}(i \alpha, 2\xi+2m_n) \subset  \D
$$
for all large $n$. 
On the other hand, since  the sequence $\{2\pi i/\lam(\beta_n)-m_n\}$ 
converges to $\xi$ as $n \to \fty$, 
we have 
$$
\left|\frac{4\pi i}{\lam(\beta_n)} -(2\xi+2m_n)\right| < \ep/4
$$
for all large $n$. 
Therefore 
$$
B_{\ep/4}\left(i \alpha, \frac{4\pi i}{\lam(\beta_n)}\right) \subset  \D
$$
hold for all large $n$. 
Thus we obtain 
$B_{\ep/4}(\alp) \subset  \inte(\LL(\beta_n))$ for all large $n$. 
\end{proof}

\begin{proof}[Proof of Theorem \ref{tan}] 
We need to prove the following four conditions (H1), (H2), (C1) and (C2), 
where 
(H1) and (H2) are corresponding to the Hausdorff convergence and 
(C1) and (C2) are corresponding to the Carath\'{e}odory  convergence: 
\begin{description}
\item[(H1)] For any $\alpha \in i  \M({2\xi})$ there exists a sequence  
$\alpha_n \in \LL(\beta_n)$ such that $\alpha_n \to \alpha$.  
\item[(H2)] If $\alpha_{n_j} \in \LL(\beta_{n_j})$ and $\alpha_{n_j} \to \alpha$ 
then $\alpha \in i\M({2\xi})$.  
\item[(C1)] For any compact subset 
$X \subset \inte(i \M({2\xi}))$, $X \subset \inte(\LL(\beta_n))$ for all large $n$. 
\item[(C2)] If there exist an open subset $O \subset \C$ 
and a infinite sequence $\{n_j\}_{j=1}^\fty$ such that 
$O \subset \inte(\LL(\beta_{n_j}))$, then $O \subset \inte(i\M({2\xi}))$. 
\end{description}

Proof of (H1):  
For any $\alp \in i\M({2\xi})$,  there exists a sequence $\{\alp(j)\}_{j =1}^\fty$ 
in the interior of $i\M({2\xi})$ such that $\alp(j) \to \alp$ as $j \to \fty$. 
It follows from Lemma \ref{tanlem} that for every $j$,  
there exists positive constant $N(j)$ such that $\alp(j) \in \LL(\beta_n)$ 
for all $n \ge N(j)$. 
Thus we obtain the result. 
To be more precise, let us choose $\{N(j)\}_{j=1}^\fty$ so  that $N(j+1)>N(j)$ and 
$N(j) \to \fty$ as $j \to \fty$, and set $\alp_n:=\alp(j)$ for every  $N(j)  \le  n  < N(j)$. 
Since $j \to \fty$ as $n \to \fty$,   
we obtain $\alp_n \in  \LL(\beta_n)$ and $\alp_n \to \alp$ as $n \to \fty$. 

\bigskip

Proof of (C1):  
Let $X$ be a compact subset of $\inte(i \M({2\xi}))$. 
For every $\alp \in X$, it follows from Lemma \ref{tanlem} that there exist $\ep(\alp)>0$ and 
$N(\alp)>0$ such that 
$B_{\ep(\alp)}(\alp) \subset \inte(\LL(\beta_n))$ 
for all $n \ge N(\alp)$. 
Since 
$$
\bigcup_{\alp \in X} B_{\ep(\alp)}(\alp)
$$
is an open covering of the compact set $X$, 
we can choose finite set of points  $\{\alp_j\}_{j=1}^l \subset X$ such that 
$$
\bigcup_{1 \le j \le l} B_{\ep(\alp_j)}(\alp_j)
$$
is also an open covering of $X$. 
Set  $N:=\max_{1 \le j \le l}{N(\alp_j)}$. 
Then 
$B_{\ep(\alp_j)}(\alp_j) \subset \inte(\LL(\beta_n))$ 
for all $n \ge N$ and all $1 \le j \le l$. 
Thus we obtain 
$X \subset \inte(\LL(\beta_n))$ for all $n \ge N$. 

\bigskip

Proof of (H2):  
For simplicity, we denote $\{n_j\}$ by $\{n\}$ and assume that 
$\alpha_{n} \in \LL(\beta_{n})$ converge to $\alpha$ as $n \to \fty$. 
Take $\rho_n \in AH(N,P) \cap \Omega$ such that $\Tr(\rho_n)=(\alp_n,\beta_n)$. 
Since $\alp_n \to \alp$ and $\beta_n \to 2$, 
the sequence $\{\rho_n\}_{n=1}^\fty$ 
converges algebraically to the conjugacy class of $\sigma_{i\alp}$ in $AH(N,P)$. 
We may assume that the representatives of the conjugacy classes 
$\rho_n$, which are also denoted by $\rho_n$, converge algebraically to $\sigma_{i\alp}$.  

We now consider representations $\chi_n$ of $\pi_1(\hN)=\la a,b,c : [b,c]=\id \ra$ into 
$\psl $ defined by 
$$
\chi_n(a):=\rho_n(a), \quad 
\chi_n(b):=\rho_n(b), \quad 
\chi_n(c):=(\rho_n(b))^{-m_{n}}. 
$$
One can see form Theorem \ref{cyclic} that the sequence $\chi_n$ 
converges algebraically to $\chi_\fty:=\hat \sigma_{i\alp,2\xi}$, 
which is defined in 4.2.  
We now claim that $\chi_\fty=\hat \sigma_{i\alp,2\xi}$ is faithful and discrete. 
If this is true, we obtain $(i\alp,2\xi) \in \B$.  
Especially we have $\im\,\xi \ne 0$. 
It then follows from $\im\,\xi \ge 0$  that $(i\alp,2\xi) \in \A$, and thus $\alp \in i \M({2\xi})$. 
Therefore we only have to show the claim above. 

Since $\pi_1(\hN)$ is finitely generated and 
since the image $\chi_n(\pi_1(\hN))$ of $\chi_n$  is equal to the discrete group $\rho_n(\pi_1(N))$, 
it follows from the theorem due to J{\o}rgensen and Klein in \cite{JK} 
that $\chi_\fty(\pi_1(\hN))$ is discrete and that 
there exist group homomorphisms 
$$
\psi_n:\chi_\fty(\pi_1(\hN)) \to \chi_n(\pi_1(\hN))
$$ 
satisfying  $\chi_n=\psi_n  \circ \chi_\fty$. 
Now suppose for contradiction that there is a non-trivial element $g$ in $\ker \chi_\fty$. 
Then it must lie  in $\ker \chi_n$ for all $n$. 
Since $\ker \chi_n$ is normally generated by a word $ b^{m_n}c$,  
and since the word length of 
$g \in \pi_1(\hN)$ with respect to the generators $a,b,c$ is bounded, 
 we obtain a contradiction. Thus we obtain the claim.

\bigskip

Proof of (C2):  
By the same argument as in the proof for (H2), 
we have $\alp \in i\M({2\xi})$ for every $\alp \in O$. 
Therefore $O \subset  i\M({2\xi})$. 
Since $O$ is open, we have 
$O \subset  \inte(i\M({2\xi}))$. 
\end{proof}

\subsection{Main theorem for horocyclic convergence}

We now state our main theorem for linear slices $\LL(\beta_n)$  
such that $\lam(\beta_n)$ converge horocyclically to $0$. 
See Figure \ref{tr_conv}, right column. 

\begin{thm}\label{horo}
Let $\{\beta_n\}_{n=1}^\fty$ be a sequence in $\C \sm [-2,2]$ which converges to $2$ as $n \to \fty$. 
Suppose that $\lam(\beta_n)$ converge horocyclically to $0$.  
Then we have the following: 
\begin{enumerate}
\item $\LL(\beta_n)$ converge to $i\M$ in the sense of Hausdorff as $n \to \fty$. 
\item $\inte(\LL(\beta_n))$ converge to $\inte(i \M)$ in the sense of Carath\'{e}odory as $n \to \fty$. 
\end{enumerate}
\end{thm}

The following lemma is an essential part of the proof of Theorem \ref{horo}. 

\begin{lem}\label{horolem}
Under the same assumption as in Theorem \ref{horo}, we have the following: 
For any $\alp \in \inte(i \M)$ there exists $\ep>0$ such that 
$B_\ep(\alp) \subset  \inte(\LL(\beta_n))$
for all large $n$. 
\end{lem}

\begin{proof} 
Let $\ep_0>0$ be the constant in Proposition \ref{ball2} for $\mu=i\alp \in\inte(\M)$. 
Let us take $0<\ep<\ep_0$ such that $B_\ep(i\alp) \subset \inte(\M)$. 
By Theorem \ref{BromLem}, one see that  there exists  $I>0$ such that $B_{\ep,I}(i\alp) \subset \A$. 
Then by Proposition \ref{ball2}, 
if we choose $I>0$ sufficiently large,  we have 
$B_{\ep/2,2I}(i \alpha) \subset \D$. 
Since $\lam(\beta_n) \to 0$ horocyclically, $\im(4\pi i/\lam(\beta_n))>2I$ for all large $n$. 
Thus, for every $\alp' \in B_{\ep/2}(\alp)$,  we have $(i\alp', 4\pi i/\lam(\beta_n)) \in \D$, 
or $(\alp', \beta_n) \in \D_\tr$. 
Therefore we obtain 
$B_{\ep/2}(\alp) \subset  \inte(\LL(\beta_n))$ for all large $n$. 
\end{proof}

\begin{proof}[Proof of Theorem \ref{horo}]
The proof is almost parallel to that of Theorem \ref{tan}. 
We need to show the following four conditions: 
\begin{description}
\item[(H1)] 
For any $\alpha \in i  \M$ there exists $\alpha_n \in \LL(\beta_n)$ such that 
$\alpha_n \to \alpha$. 
\item[(H2)] If $\alpha_{n_j} \in \LL(\beta_{n_j})$ and $\alpha_{n_j} \to \alpha$ then 
$\alpha \in i\M$. 
\item[(C1)] For any compact subset 
$X$ in $\inte(i \M)$, $X \subset \inte(\LL(\beta_n))$ for all large $n$. 
\item[(C2)] 
If there exist an open subset $O \subset \C$ 
and a infinite sequence $\{n_j\}_{j=1}^\fty$ such that 
$O \subset \inte(\LL(\beta_{n_j}))$ then 
$O \subset \inte(i\M)$. 
\end{description}

Proof of  (H1):  
For any $\alp \in i\M$,  there exists a sequence $\alp(j) \in \inte(i\M)$ 
such that $\alp(j) \to \alp \ (j \to \fty)$. 
It follows from Lemma \ref{horolem} that for each $j$ we have $\alp(j) \in \LL(\beta_n)$ for all large $n$. 
Thus we obtain the claim. 

Proof of  (C1): 
Let $X \subset \inte(i \M)$ be a compact subset. 
For each $\alp \in X$, it follows from Lemma \ref{horolem} that there exist $\ep(\alp)>0$ and 
$N(\alp)>0$ such that $B_{\ep(\alp)}(\alp) \subset \inte(\LL(\beta_n))$ 
for all $n \ge N(\alp)$. Since 
$\bigcup_{\alp \in X} B_{\ep(\alp)}(\alp)$ is an open covering of $X$, 
we may choose a finite set of points $\{\alp_j\} \subset  X$ such that 
$\bigcup_{j} B_{\ep(\alp_j)}(\alp_j)$ is also an open covering.  
Since $B_{\ep(\alp_j)}(\alp_j) \subset \inte(\LL(\beta_n))$ for all 
$n \ge N:=\max_j{N(\alp_j)}$, we have 
$X \subset \inte(\LL(\beta_n))$ for all $n \ge N$.

Proof of  (H2): 
For simplicity we denote $\{n_j\}$ by $\{n\}$,  and assume that 
$\alpha_{n} \in \LL(\beta_{n})$ converge to $\alpha$. 
Take $\rho_n \in AH(N,P) \cap \Omega$ such that $\Tr(\rho_n)=(\alp_n,\beta_n)$. 
Since $\alp_n \to \alp$, $\beta_n \to 2$,  
the sequence $\{\rho_n\}_{n=1}^\fty$ 
converges to  $\sigma_{i\alp} \in R(N,P)$,  and 
since $AH(N,P)$ is closed, we have $\sigma_{i\alp} \in AH(N,P)$.  
Therefore we obtain $i\alp \in \M$ and hence $\alp \in i\M$.

Proof of  (C2): 
By the same argument as in (H2), we have $\alp \in i\M$ for every $\alp \in O$. 
Therefore  $O \subset  i\M$. 
Since $O$ is open, we have $O \subset  \inte(i\M)$. 
\end{proof}

\subsection{Non local connectivity}

Here we will show that there exists a linear slice which is not locally 
connected at their boundary (see Figure \ref{nonloc}). 
This is a direct consequence of Bromberg's argument in \cite{Br} 
showing that $AH(N,P)$ is not locally connected. 
This result is concerned with vertical slices of $\A$, 
whereas Theorems \ref{tan} and \ref{horo} are concerned with horizontal slices of $\A$. 

\begin{figure}
\begin{center}
\includegraphics[width=4.5cm, bb=0 0 444 450]{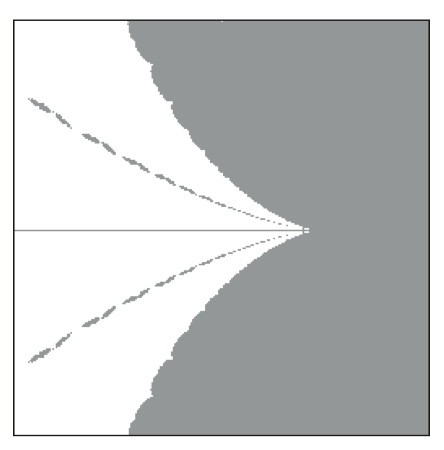}
\caption{The linear slice $\LL(\alp)$ (gray part) for $\alp=5.9+0i$; 
 restricted to a neighborhood of $2$ whose width is about $0.2$. 
 (The horizontal line thorough $2$ is the locus where the computer can not detect non-discreteness.) }
\label{nonloc}
\end{center}
\end{figure}

\begin{thm}
There exists $\alp\in \C$ such that $\LL(\alp)$ 
is not locally connected at $2 \in \bd \LL(\alp)$; 
that is,   $U \cap \LL(\alp)$ is disconnected for any sufficiently small 
neighborhood $U \subset \C$ of $2$. 
\end{thm}

\begin{proof}
The homeomorphism $F:\D_\tr^+ \to \D$ defined in section 5.3 
induces a homeomorphism from 
\begin{eqnarray*}
\LL(\alp) \cap \C_+=\{\beta \in \C : (\alp,\beta) \in \D_\tr^+\} 
\end{eqnarray*}
to a slice  
$$
\{\ze \in \hC : (i \alp, \ze) \in \D\}
$$
of $\D$. 
Since $F(\alp,2)=(i\alp,\fty)$, to show that 
$\LL(\alp)$ is not locally connected at $2$ for some $\alp$, 
it suffices  to show that the set 
$\{\ze \in \hC : (i \alp, \ze) \in \D\}$ is not locally connected at $\ze=\fty$. 
We will show this by using the fact observed in \cite{Br} that the vertical slice 
$\{\ze \in \hC : (i \alp, \ze) \in \A\}$ of $\A$ is not locally connected at $\ze=\fty$ for some $\alp$. 

From the argument of Bromberg in the proof of Theorem 4.15 in \cite{Br}, 
there exist $\mu \in \inte(\M)$, $\ze \in \C$ with $\im\,\ze>0$, and $\ep>0$ such that 
$B_\ep(\mu,\ze+2n)$ are contained in different connected components of 
$$
\{(\nu,z) \in \A : |\nu-\mu|<2\ep\}
$$ for every integer $n$. 
By Theorem \ref{BromThm}, we can take neighborhoods $U$, $W$ of $(\mu,\fty)$ in $\A$, $\D$, respectively,  
such that $\varphi=F \circ \Tr \circ \Phi:U \to W$ is a homeomorphism. 
We may assume that 
$U$ is of the form 
$$
U=\A \cap \{(\nu,z) \in \C^2 : |\nu-\mu|<2\ep, |z|>K\}.
$$
Then for all large $n$, 
$B_\ep(\mu,\ze+2n)$ are contained in $U$, and thus 
contained in distinct connected components of $U$.

By choosing $\ep>0$ sufficiently small and  $K>0$ sufficiently large, 
we see from Proposition \ref{ball1}  that 
$B_{\ep/2}(\mu,\ze+2n) \subset \D$ for all large $n$. 
Therefore, $B_{\ep/2}(\mu,\ze+2n)$ are contained in distinct connected components of $W$ for all large $n$.  
Since $W \subset \D$ is a neighborhood of $(\mu,\fty)$, 
we see  that  the set $\{\ze \in \hC : (\mu, \ze) \in \D\}$ 
 is not locally connected at $\ze=\fty$.  
 Letting $\alp=-i\mu$, we obtain the result.   
\end{proof}

\section{Complex Fenchel-Nielsen coordinates}

In this section, we restate Theorems \ref{tan} and \ref{horo} in terms of 
the complex Fenchel-Nielsen coordinates. 
We begin with recalling  the definition of the real Fenchel-Nielsen coordinates for 
Fuchsian representations. 

Given $\lam>0$, we define a representation $\eta_\lam \in R(N,P)$  by 
$$
\eta_\lam(a):=\frac{1}{\sinh (\lam/2)}\left(\begin{array}{cc}
\cosh(\lam/2)  & -1 \\ 
-1 & \cosh(\lam/2)
 \end{array}\right), \quad 
\eta_\lam(b):=\left(\begin{array}{cc}
e^{\lam/2}  & 0 \\ 
0 & e^{-\lam/2}
 \end{array}\right). 
$$
Then $\eta_{\lam}(\pi_1(N))$ acts properly discontinuously on the upper-half plane $\HH^2$, 
and hence is a Fuchsian group (see Figure \ref{domain}, left). 
Note that $\eta_\lam(a)$ fixes $-1, 1$,  $\eta_\lam(b)$ fixes $0, \fty$,  
and thus the axes of $\eta_\lam(a)$ and $\eta_\lam(b)$ are perpendicular to each other.   
In addition, the complex length of $\eta_\lam(b)$ is equal to $\lam \in \R$.

 Now we add a twisting parameter $\tau$. 
Given $(\lam,\tau) \in \R_+ \times \R$, 
we define a fuchsian representation 
$\eta_{\lam,\tau} \in R(N,P)$
by 
$$
\eta_{\lam,\tau}(a):=\left(\begin{array}{cc}
e^{\tau/2}  & 0 \\ 
0 & e^{-\tau/2}
 \end{array}\right)\eta_{\lam}(a), \quad 
 \eta_{\lam,\tau}(b)
:=\eta_{\lam}(b). 
$$
Note that the quotient surface $\HH^2/\eta_{\lam,\tau}(\pi_1(N))$ is obtained by 
cutting the surface $\HH^2/\eta_\lam(\pi_1(N))$ along the geodesic representative of $\eta_\lam(b)$, 
twisting by hyperbolic length $\tau$ and re-glueing (see Figure \ref{domain}, right). 

\begin{figure}
\begin{center}
\includegraphics[height=4cm, bb=0 0 1304 394]{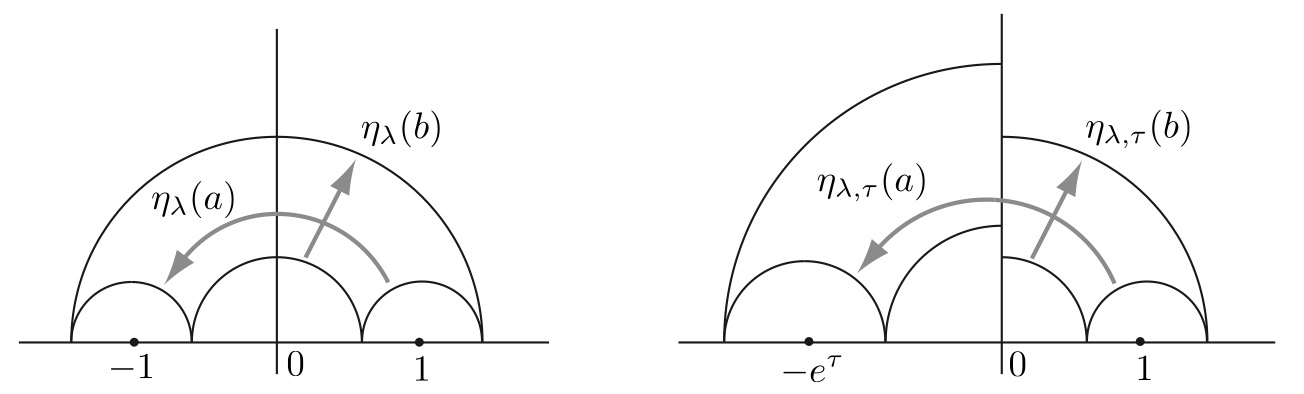}
\caption{Fundamental domains of images of  $\eta_\lam$ (left) and $\eta_{\lam,\tau}$ (right). }
\label{domain}
\end{center}
\end{figure}

Now we obtain a map 
$$
FN:\R_+ \times \R \to  R(N,P) 
$$
defined by $(\lam,\tau) \mapsto \eta_{\lam,\tau}$. 
It is well-known that this map is a homeomorphism onto the space of Fuchsian representations. 
By allowing the parameters $\lam,\tau$ to be complex numbers, 
we obtain a map
$$
FN:(\C \sm 2\pi i \Z) \times \C \to R(N,P). 
$$
We say that $(\lam,\tau)$ is the {\it complex Fenchel-Nielsen coordinates} of the representation 
$\eta_{\lam,\tau}$.  
Note that if $\lam \in \R$ and $\tau \in \C$, $\eta_{\lam,\tau}$ is the complex earthquake of $\eta_\lam$,  
see \cite{Mc2}. 
It is known by  Kourouniotis \cite{Ko}  and Tan \cite{Ta} 
that there is an open subset of  $(\C \sm 2\pi i \Z) \times \C$ containing 
$\R_+ \times \R$ such that the map $FN$ induces a homeomorphism 
from this set onto the quasifuchsian space $MP(N,P)$. 

Let 
$$
\D_{FN}:=\{(\lam,\tau) \in (\C \sm 2\pi i \Z) \times \C : \eta_{\lam,\tau} \in AH(N,P)\}. 
$$
Since we have 
\begin{eqnarray*}
\tr^2 \eta_{\lam,\tau}(a)
&=& 4 \coth^2 \left(\frac{\lam}{2}\right)\cosh^2\left(\frac{\tau}{2}\right), \\
\tr^2 \eta_{\lam,\tau}(b)&=&4 \cosh^2 \left(\frac{\lam}{2}\right), 
\end{eqnarray*}
the map $\Theta: (\C \sm 2\pi i \Z) \times \C \to \C^2$ defined by 
$$
\Theta(\lam,\tau):=\left(
2\coth \left( \frac{\lam}{2} \right)  \cosh \left( \frac{\tau}{2} \right),  
2 \cosh \left( \frac{\lam}{2} \right) 
\right)
$$ 
takes $\D_{FN}$ onto $\D_\tr$. 
For a given $\lam \in \C \sm 2\pi i \Z$, let 
$$
\wt \LL(\lam):=\{\tau \in \C: \eta_{\lam,\tau} \in AH(N,P)\}.  
$$
We define a map $f_\lam:\C \to \C$ by 
$$
f_\lam(z):=2 \coth\left(\frac{\lam}{2}\right)\cosh\left(\frac{z}{2}\right) 
$$
so that we have  $\Theta(\lam,\tau)=(f_\lam(\tau), 2\cosh(\lam/2))$. 
Then the map $f_\lam$ takes 
$\wt \LL(\lam)$ onto  $\LL(\beta)$ where $\beta=2\cosh (\lam/2)$. 
Note that $\wt \LL(\lam)$ is $\la z+\lam, z+2 \pi i \ra$-invariant,  where the translation 
$z \mapsto z+\lam$ corresponds to the Dehn twist about $b$. 

We want to understand the shape of $\wt \LL(\lam)$ by using the Maskit slice $\M$ 
 when $\lam$  lies in  $\C_+$ and is close to zero. 
 To this end,  we normalize $\wt \LL(\lam)$ so that 
 the action of the Dehn twist  about $b$ corresponds to the translation $z \mapsto z+2$.  
 (Recall that the Maskit slice $\M$ has this property.)  
 Let us define a map  $g_\lam:\C \to \C$ by
$$
g_\lam(z):=\frac{2}{\lam}(z-\pi i) 
$$ 
 and set 
 $$
\wh \LL(\lam):=g_\lam (\wt \LL(\lam))
$$
Then $\wh \LL(\lam)$ is $\la z+2, z+4 \pi i/\lam \ra$-invariant and the map 
$$
h_\lam(z):=f_\lam \circ g_\lam^{-1}(z)
$$
takes zero to zero and $\wh \LL(\lam)$ onto $\LL(\beta)$, where $\beta=2\cosh(\lam/2)$. 

Since 
\begin{eqnarray*}
h_\lam(z)
&=&2\coth\left(\frac{\lam}{2}\right)\cosh\left(\frac{\lam z}{4}+\frac{\pi i}{2}\right) \\
&=&2i \coth\left(\frac{\lam}{2}\right)\sinh\left(\frac{\lam z}{4}\right), 
\end{eqnarray*}
 one can see that  if $\lam_n \to 0$ as $n \to \fty$, 
then $h_{\lam_n}(z) \to iz$ uniformly on any compact subset of $\C$.   
Thus we obtain the following corollary of Theorems \ref{tan} and \ref{horo} (see Figure \ref{tau_conv}):

\begin{cor}
Suppose that  $\lam_n \in \C_+, \lam_n \to 0$ as $n \to \fty$. 
\begin{enumerate}
\item If $\lam_n \to 0$ horocyclically, then 
$\wh \LL(\lam_n)$ converge to $\M$ 
in the sense of Hausdorff, 
and 
$\inte(\wh \LL(\lam_n))$ converge to $\inte(\M)$ in the sense of Carath\'{e}odory. 

\item Suppose that $\lam_n \to 0$ tangentially.  
In addition we assume that 
there exist a sequence of integers 
$\{m_n\}_{n=1}^\fty$ such that the sequence 
$2\pi i/\lam_n-m_n$ converges to some $\xi \in \C$ as $n \to \fty$. Then  
$\wh \LL(\lam_n)$ converge to $\M({2\xi})$ 
in the sense of Hausdorff, 
and 
$\inte(\wh \LL(\lam_n))$ converge to $\inte(\M({2\xi}))$ in the sense of Carath\'{e}odory. 
\end{enumerate}
\end{cor}

\begin{proof}
The statement for Hausdorff convergence can be easily seen. 
The statement for Carath\'{e}odory  convergence follows form 
Hausdorff convergence of the complements. 
\end{proof}

\begin{figure}
\begin{center}
\includegraphics[height=15cm, bb=0 0 651 1018]{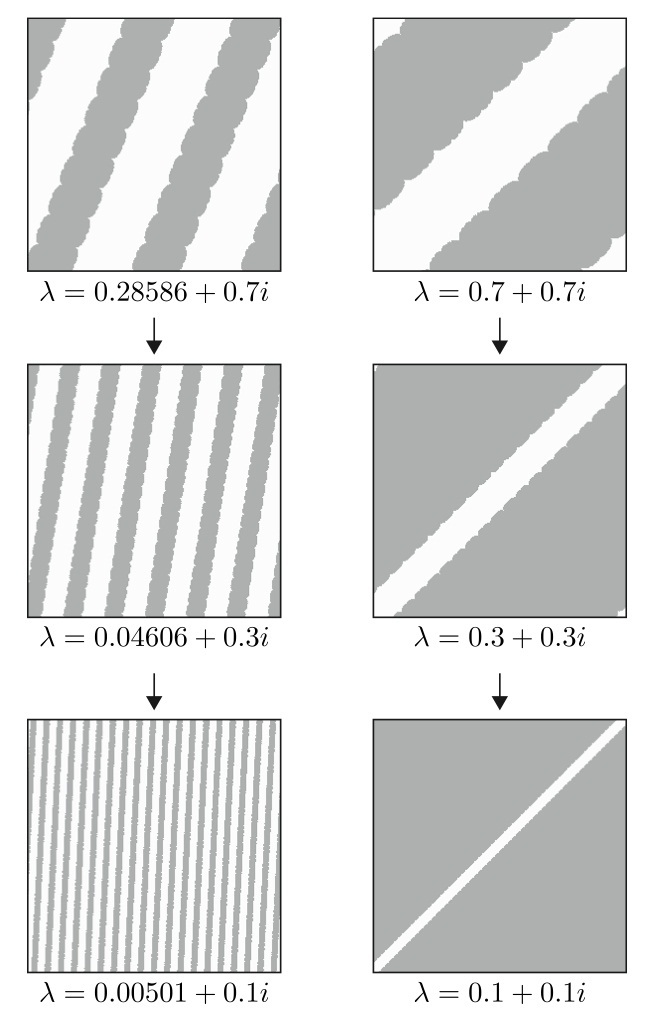}
\caption{Computer-generated figure of  $\wt \LL(\lam)$
 (gray parts) for $\lam$ close to $0$, restricted to the square 
of width  $2\pi$ centered at $\pi i$.  
The left column corresponds to tangential convergence $\lam \to0$, where 
$\lam$ are points on  the circle $|z-1|=1$ 
whose imaginary part equal $0.7$ (top),  $0.3$ (middle) and $0.1$ (bottom). 
The right column corresponds to horocyclic convergence $\lam \to 0$, 
where $\lam$ equal $0.7+0.7i$ (top),  $0.3+0.3i$ (middle) and $0.1+0.1i$ (bottom).  }
\label{tau_conv}
\end{center}
\end{figure}

\end{document}